\documentclass[12pt]{article}
\usepackage{amsmath,amssymb,amsfonts,amsthm}
\usepackage[textwidth=6.5in,textheight=9in,left=0.75in,right=0.75in,top=0.75in,bottom=0.75in]{geometry}
\usepackage{color}
\usepackage{verbatim}
\usepackage{epsfig,subfigure,epstopdf}
\usepackage[]{graphics}
\usepackage{graphicx,inputenc}
\usepackage{subfigure}
\usepackage[justification=centering]{caption}
\usepackage{pictex}
\usepackage{wrapfig}
\usepackage[T1]{fontenc}
\usepackage{authblk}

\usepackage[colorlinks,linktocpage,linkcolor=blue]{hyperref}
\usepackage{fancyhdr}

\newtheorem{definition}{Definition}[section]
\newtheorem{lemma}{Lemma}[section]
\newtheorem{theorem}{Theorem}[section]
\newtheorem{corollary}{Corollary}[section]

\newtheorem{remark}{Remark}[section]
\numberwithin{equation}{section}

\def\D{^C\!D_t^{\alpha}}
\def\Dtau{^C\!D_\tau^{\alpha}}
\def\N+{n\in\mathbb{N}^{+}}

\def\n{\partial{\overrightarrow{\bf n}}}

\begin{document}

\title{Recovering an Unknown Source in a Fractional Diffusion Problem}
\author[1]{William Rundell\thanks{rundell@math.tamu.edu}} 
\author[2]{Zhidong Zhang
\thanks{zhidong.zhang@helsinki.fi}}
\affil[1]{Department of Mathematics, Texas A\&M University, USA}
\affil[2]{Department of Mathematics and Statistics, University of Helsinki, Finland}

\maketitle

\begin{abstract}
A standard inverse problem is to determine a source which
is supported in an unknown domain $D$ from external boundary measurements.
Here we consider the case of a time-dependent situation where the source is
equal to unity in an unknown subdomain $D$ of a larger given domain $\Omega$.
Overposed measurements consist of time traces of the solution
or its flux values on a set of discrete points on the boundary
$\partial\Omega$.
The case of a parabolic equation was considered in \cite{HettlichRundell:2001}.
In our situation we extend this to cover the {\it subdiffusion\/} case
based on an anomalous diffusion model and leading to a fractional order
differential operator.
We will show a uniqueness result and examine a reconstruction algorithm.
One of the main motives for this work is to examine the dependence of the
reconstructions on the parameter $\alpha$, the exponent of the fractional
operator which controls the degree of anomalous behavior of the process.
Some previous inverse problems based on fractional diffusion models
have shown considerable differences between classical Brownian diffusion
and the anomalous case.
\end{abstract}

\section{Introduction}

Our aim is to recover the location and shape of an extended source function
$F = \chi(D)$ in a diffusion problem from making time-trace boundary
measurements,
\begin{equation}\label{FDE}
\begin{cases}
\begin{aligned}
\D u-\triangle u&=\chi_{{}_D}, &&(x,t)\in \Omega \times (0,T);\\
u(x,0)&=0, &&x\in \Omega; \\
u(x,t)&=0, &&(x,t)\in \partial\Omega\times(0,T).
\end{aligned}
\end{cases}
\end{equation}
$\Omega\subseteq \mathbb{R}^2$ is the unit disc, $\chi_{{}_D}$ is the
characteristic function on $D$ which is the source domain we need to
recover with $\overline{D}\subseteq \Omega$.
The overposed data is a time trace of the flux at a (small) finite number
$m$ of points located on the boundary $\partial D$, 
\begin{equation*}
\frac{\partial u}{\n}(z_\ell,t)=g_\ell(t),\ t\in[0,T], \  \ell=1,\dots, m. 
\end{equation*} 
In this paper, we restrict the set of admissible boundaries to be
star-like domains with respect to a point within $\Omega$, 
$$
\partial D=\{q(\theta)(\cos{\theta},\sin{\theta})^{\top}:\theta\in[0,2\pi]\}
$$
with a smooth, periodic function $0<q(\theta)<1$.
In equation~ \eqref{FDE}  $\D$ denotes the Djrbashian-Caputo fractional derivative of
order $\alpha$, $0<\alpha<1$ which will be defined in the next section.

We have described \eqref{FDE} in the simplest setting in
the sense we have taken the exterior boundary to be the unit circle
and have chosen homogeneous initial and boundary data.
This simplifies the exposition and, in particular, many of the representation
formulae.
Adding in nonhomogeneous initial/boundary conditions: $u(x,0)=u_0(x)$
and $u(x,t) = f(x,t)$ for $x$ on $\partial\Omega$ and sufficiently smooth $f$,
would be completely straightforward.
We could also have assumed a source of the form $a(t)\chi(D)$ where
$a(t) \in L^\infty(0,\infty)$ is known.
In each of these cases no technical issues would ensue or changes to
the main results.
Taking $\Omega$ to be a simply connected domain with $C^2$ boundary
$\partial\Omega$ is also possible in theory but we have used the specific
eigenfunction expansion for $-\triangle$ for a circle in both the uniqueness
result and the reconstruction algorithm.
The key change would be to equations
\eqref{eqn:theta_cond} and \eqref{eqn:m-theta_cond}
where the trigonometric function  would have to be replaced 
by the values of the Laplace eigenfunction for $\Omega$ evaluated on
$\partial\Omega$.
While these share the same properties when $\Omega$ is the unit circle,
this extension would require some further analysis.

The model \eqref{FDE} represents a so-called {\it anomalous diffusion\/}
process generalizing classical, Brownian  diffusion based on the heat equation.
This latter model can be viewed as a random walk
in which the dynamics are governed by an uncorrelated,
Markovian, Gaussian stochastic process.
The key assumption is that a change in the direction
of motion of a particle is random and that the mean-squared displacement
over many changes is proportional to time, i.e. $\langle x^2\rangle = C t$.
This easily leads to the derivation of the underlying differential equation
being the heat equation.
On the other hand, when the random walk involves correlations,
non-Gaussian statistics or a non-Markovian process
(for example, due to ``memory'' effects)
the classical  diffusion equation will fail to describe the macroscopic limit.
For example, if we replace the space-time correlation by 
$\langle x^2\rangle = C t^\alpha$ then
it can be shown that this
leads to a {\it subdiffusive\/} process and, importantly leads to a
tractable model where the partial differential equation is replaced by
the nonlocal equation \eqref{FDE}.

This paper is a generalisation of \cite{HettlichRundell:2001} where
the same problem was considered for the classical parabolic case, $\alpha=1$.
Our approach will be  the same, but here we must deal with the technical
issues of replacing the far simpler classical time derivative by the
nonlocal operator ${}_a^C\! D_t^\alpha$.
Thus while in the case $\alpha=1$ \eqref{FDE} is pointwise defined and
the Markovian property dictates that for any time step $t$ the solution can be
uniquely obtained from any single previous step $t-\delta t$, this is far
from the case if $\alpha<1$ where the complete time history of the function
$u$ has to be retained in the evolution.
In some previous cases involving fractional derivatives the inverse problem
has very different properties, especially with respect to degree of
ill-conditioning, from the classical case, see \cite{JinRundell:2015}
for an overview.
The poster child here is the backward diffusion problem.
This is severely ill-conditioned for the heat equation, but for $0<\alpha<1$
is only moderately so  (equal to a 2-derivative loss) ,
\cite{ChengNakagawaYamamotoYamazaki:2009}.
Thus an important aspect of our studies here is to determine, if any,
the differences made by the anomalous diffusion operator from that of
the classical one.
We will also investigate the influence of the number $m$ of measurement
points on both the question of uniqueness and reconstruction.

\section{Preliminary material}

\subsection{Fractional derivatives}
The (left-sided) fractional integral of order $\alpha$ is defined for
$f\in  L^1(a,b)$ by
\begin{equation}\label{eqn:RL-int}
({_aI_x^\alpha} f)(x)=\frac{1}{\Gamma(\alpha)}\int_a^x(x-s)^{\alpha-1}f(s)\,ds,
\end{equation}
and leads naturally to a fractional derivative in one of two ways.
The (left-sided) Riemann-Liouville fractional derivative of order $0<\alpha<1$,
is defined by
\begin{equation*}
{}_a^R\! D_t^\alpha  f(t) := 
\frac{1}{\Gamma(1-\alpha)}\frac{d\ }{dt}\int_a^t (t-s)^{-\alpha}f(s)\,ds,
\end{equation*}
and the (left-sided)  Djrbashian-Caputo fractional derivative of order $\alpha$ by
\begin{equation*}
{}_a^C\! D_t^\alpha  f(t) := 
  \frac{1}{\Gamma(1-\alpha)}\int_a^t(t-s)^{-\alpha}f'(s)\,ds.
\end{equation*}
In both cases note the specific dependence on the endpoint $a$.
Some references are 
\cite{Djrbashian:1966,Djrbashian:1993,
Caputo:1967,
SakamotoYamamoto:2011,
SamkoKilbasMarichev:1993}.

The Djrbashian-Caputo derivative is more restrictive than the Riemann-Liouville
since it requires the classical derivative to be absolutely integrable and
we implicitly assume that this condition holds.
Generally, the Riemann-Liouville and Djrbashian-Caputo derivatives are
different, even when both derivatives are defined, and we only have to consider
the constant function to see this.
Nonetheless, as we must expect, they are closely related to each other and
under the assumption that the function to which they are applied vanishes
at the starting point they are equal.
Thus in~\eqref{FDE} as stated we could have equally replaced 
${}_0^C\! D_t^\alpha$ by ${}_0^R\! D_t^\alpha$.
However, in the face of a non-homogeneous initial condition the regularity
of the solution of the direct problem for \eqref{FDE} would change.

\subsection{Mittag-Leffler function}
\par This function plays a central role in
fractional diffusion equations. It is a two-parameter function defined as 
\begin{equation*}
E_{\alpha,\beta}(z) = \sum_{k=0}^\infty \frac{z^k}{\Gamma(k\alpha+\beta)},
\ z\in \mathbb{C}.
\end{equation*}
The Mittag-Leffler function generalizes the exponential
function since $E_{1,1}(z)=e^z$ and as $\alpha\to 1$
the fractional diffusion process recovers classical diffusion as described
by the heat equation.
The following property will be used later.
The proof can be found in standard references, for example,
\cite[Lemma $3.2$]{SakamotoYamamoto:2011}.

\begin{lemma}\label{mittag_derivative}
	\par For $\lambda>0,\ \alpha>0$ and $\N+,$ we have 
	$$\frac{d^n}{dt^n}E_{\alpha,1}(-\lambda t^\alpha)
	=-\lambda t^{\alpha-n}E_{\alpha,\alpha-n+1}(-\lambda t^\alpha),
	\ t>0.$$
In particular, 
$\frac{d}{dt}E_{\alpha,1}(-\lambda t^\alpha) =
-\lambda t^{\alpha-1}E_{\alpha,\alpha}(-\lambda t^\alpha)$,
$\;t>0$.
\end{lemma}

\subsection{The direct problem for equation \eqref{FDE}}
\par For the unit disc $\Omega,$ denote the  eigensystem of the Laplacian $-\triangle$ 
with the Dirichlet boundary condition by $\{(\lambda_n, \varphi_n(x)): \N+\}.$ Here, $\{\lambda_n:\N +\}$ is 
indexed by nondecreasing order and strictly positive, and  $\{\varphi_n(x):\N +\}$ constitutes an orthonormal basis in $L^2(\Omega).$ 
The polar representation of $\varphi_n$ is 
\begin{equation}\label{eigenfunction}
\varphi_n(r,\theta)=w_nJ_m(\sqrt{\lambda_n}r)\cos{(m\theta+\phi_n)}, 
\end{equation} 
where $m=m(n)$, the phase $\phi_n$ is either $0$ or $\pi/2$ and $w_n$ is the normalized weight factor.
Here $J_m(z)$ is the first kind Bessel function with degree $m$. 

With the above, \cite{SakamotoYamamoto:2011}
gives the following theorem for the direct problem of \eqref{FDE}.   
Here $H^k(\Omega)$ are the usual Sobolev spaces.

\begin{theorem}\label{direct}
	There exists a unique weak solution $u\in L^2(0,T;H^2(\Omega)\cap H_0^1(\Omega))$ of \eqref{FDE} with the representation 
	\begin{equation}\label{solution}
	u(x,t)=\sum_{n=1}^\infty\left(\int_0^t \int_{D} \varphi_n(y) (t-\tau)^{\alpha-1}
	E_{\alpha,\alpha}(-\lambda_n(t-\tau)^\alpha)\ {\rm d}y\ {\rm d}\tau \right)\varphi_n(x)
	\end{equation}
	and the regularity estimate
	\begin{equation*}
	\|u\|_{L^2(0,T;H^2(\Omega))}+\|\D\|_{L^2(\Omega\times(0,T))}\le C(T,D), 
	\end{equation*}
	where the notation $C(T,D)$ indicates the dependence on the final time $T$ and the domain $D$.
\end{theorem}

\begin{proof}
This theorem is a specific case of \cite[Theorem 2.2]{SakamotoYamamoto:2011}
based on the fact that the source term is independent of $t$.
See a later remark about generalizing the situation in \eqref{FDE}
to include a known time-dependent factor in the source term.
\end{proof}	

\section{Main results} 

\par In this section we will prove the main theoretical result:
under suitable restrictions,  two observation points are sufficient to
determine the internal domain $D$ uniquely.

\subsection{Harmonic basis}
Let $\xi_m^{c,s}(r,\theta)=\frac{1}{\pi}r^m\{\cos{m\theta},\sin{m\theta}:m\in
\mathbb{N}\}$ denote the set of harmonic functions in $\Omega$.
With the given normalization it forms
a complete orthonormal basis in $L^2 (\partial\Omega)$.
First, we show that this basis can be used to gain a convergent approximation
to the flux data $\frac{\partial u}{\n}(z_\ell,t).$

Define the smooth approximation $\psi_\ell^M\in C^\infty(\overline{\Omega})$
of the delta distribution at $z_\ell$ as 
\begin{equation*}
\psi_\ell^M (x)=\sum_{m=1}^M \xi_m^c(z_\ell)\xi_m^c(x)+\xi_m^s(z_\ell)\xi_m^s(x),\ \ell=1,2,
\end{equation*} 
then the set $u_\ell^M$ are weak solutions of the FDEs
\begin{equation*}
\begin{cases}
\begin{aligned}
\D u_\ell^M -\triangle u_\ell^M&=0, &&(x,t)\in \Omega\times (0,T);\\
u_\ell^M&=0, &&(x,t)\in \partial \Omega\times (0,T);\\
u_\ell^M&=-\psi_\ell^M, &&(x,t)\in \Omega\times\{0\}.
\end{aligned}
\end{cases}
\end{equation*} 
It follows from \cite{SakamotoYamamoto:2011} that we have the 
regularity results 
$u_\ell^M \in C((0,T]; H^2(\Omega)\cap H_0^1(\Omega))$, 
$\D u_\ell^M \in C((0,T]; L^2(\Omega))$.

\begin{lemma}\label{wlM}
Define $w_\ell^M=u_\ell^M+\psi_\ell^M,$ then $w_\ell^M \in C((0,T]; H^2(\Omega)\cap H_0^1(\Omega)),$ 
$\D w_\ell^M \in C((0,T]; L^2(\Omega))$ 
and 
$$
\lim_{M\to \infty}\int_D w_\ell^M(x,t) {\rm d}x=-\frac{\partial u}{\n} (z_\ell),\quad
 \ell=1,2.
$$ 
\end{lemma}

\begin{proof}
The regularity follows from those of $u_\ell^M$ and $\psi_\ell^M$. 
Since $\psi_\ell^M$ are linear combinations of harmonic functions,
they satisfy the equations $\D \psi_\ell^M -\triangle \psi_\ell^M=0, \ \ell=1,2$.
Hence, $w_\ell^M,\ \ell=1,2$ are weak solutions of
$\D w_\ell^M -\triangle w_\ell^M=0$, $(x,t)\in \Omega\times (0,T)$ 
subject to the boundary condition $w_\ell^M|_{\partial \Omega}=\psi_\ell^M$
and the initial condition $w_\ell^M(\cdot,0)=0$.
Then for each $v\in L^2(0,T;H_0^1(\Omega))$,  
\begin{equation}\label{equality_1}
\int_0^t \int_{\Omega}  (\Dtau w_\ell^M) v +
\nabla w_\ell^M\! \cdot\! \nabla v \,{\rm d}x\,{\rm d}\tau=0.
\end{equation}

A direct calculation gives
\begin{equation*}
\begin{aligned}
\int_0^t \int_D w_\ell^M(x,\tau)\ {\rm d}x\ {\rm d}\tau 
&=\int_0^t \int_D w_\ell^M(x,t-\tau)\ {\rm d}x\ {\rm d}\tau\\
&= \int_0^t \int_\Omega [\D u(x,\tau)-\triangle u(x,\tau)]w_\ell^M(x,t-\tau)\ {\rm d}x\ {\rm d}\tau\\
&:=I_1+I_2.
\end{aligned}
\end{equation*}
For $I_1,$ by the regularity of the functions $w_\ell^M$ and $u$,
it holds that 
\begin{equation*}
\begin{aligned}
I_1&=\int_0^t \int_\Omega {}\D u(x,\tau)w_\ell^M(x,t-\tau)\ {\rm d}x\ {\rm d}\tau
=\int_{\Omega} {}\D u(x,t)*w_\ell^M(x,t)\ {\rm d}x\\
& =\int_{\Omega} \frac{t^{-\alpha}}{\Gamma(1-\alpha)}*\frac{\partial u}{\partial t}(x,t)*w_\ell^M(x,t)\ {\rm d}x 
= \int_{\Omega} \frac{t^{-\alpha}}{\Gamma(1-\alpha)}*\left(\frac{\partial u}{\partial t}(x,t)*w_\ell^M(x,t)\right)\ {\rm d}x, 
\end{aligned}
\end{equation*}
where $*$ represents the convolution in $t$.
Due to the zero initial conditions of $u$ and $w_\ell^M$,
we have
$$\frac{\partial u}{\partial t}(x,t)*w_\ell^M(x,t)=u(x,t)*\frac{\partial w_\ell^M}{\partial t}(x,t).$$
Hence, 
\begin{equation*}
\begin{aligned}
I_1&=\int_{\Omega} \frac{t^{-\alpha}}{\Gamma(1-\alpha)}*\frac{\partial w_\ell^M}{\partial t}(x,t)*u(x,t)\ {\rm d}x
=\int_{\Omega} {}\D w_\ell^M *u(x,t)\ {\rm d}x\\
&=\int_0^t \int_{\Omega} {}\D w_\ell^M(x,t-\tau) 
u(x,\tau) \ {\rm d}x\ {\rm d}\tau.
\end{aligned}
\end{equation*}
For the term $I_2$,
Green's first formula and the boundary condition of $w_\ell^M$ give that  
\begin{equation*}
\begin{aligned}
I_2&=\int_0^t \int_\Omega-\triangle u(x,\tau)w_\ell^M(x,t-\tau)\ {\rm d}x\ {\rm d}\tau\\
&=\int_0^t \int_{\Omega} \nabla u(x,\tau)\cdot  \nabla w_\ell^M(x,t-\tau)
\ {\rm d}x\ {\rm d}\tau 
-\int_0^t \int_{\partial\Omega} \frac{\partial u}{\n}(x,\tau) \psi_\ell^M(x)
\ {\rm d}x\ {\rm d}\tau. 	
\end{aligned}
\end{equation*}
The results of $I_1$ and $I_2$,
\eqref{equality_1} and the definition of $\psi_\ell^M$ now show that 
\begin{equation*}
\begin{aligned}
\int_0^t\!\int_D\! w_\ell^M(x,\tau)\ {\rm d}x\, {\rm d}\tau
=&\int_0^t\!\int_{\Omega} \bigl[{}\D w_\ell^M(x,t-\tau) u(x,\tau)\; + \\
&\qquad\qquad\qquad \nabla w_\ell^M(x,t-\tau)\cdot\!\nabla u(x,\tau)\bigr]\,{\rm d}x\,{\rm d}\tau \\
\qquad &-\int_0^t\!\int_{\partial\Omega} \frac{\partial u}{\n}(x,\tau) \psi_\ell^M(x)
\, {\rm d}x\, {\rm d}\tau\\
=&-\int_0^t \sum_{m=1}^M c_m^c(\tau) \xi_m^c(z_\ell)+c_m^s(\tau) \xi_m^s(z_\ell)\,
{\rm d}\tau,
\end{aligned}
\end{equation*}
where $c_m^{\{c,s\}}(\tau)$ are the Fourier coefficients of
$\frac{\partial u}{\n}(x,\tau)$ with respect to the basis 
$\{\xi_m^{\{c,s\}}(x):m\in \mathbb{N}\}$ in $L^2(\partial \Omega)$.
Taking derivative with respect to $t$ in the above yields 
\begin{equation*}
\begin{aligned}
\int_D w_\ell^M(x,t)\ {\rm d}x=
-\sum_{m=1}^M [c_m^c(t) \xi_m^c(z_\ell)+c_m^s(t) \xi_m^s(z_\ell)],
\end{aligned}
\end{equation*}
which together with the pointwise convergence of the Fourier series gives 
$$
\lim_{M\to \infty} \int_D w_\ell^M(x,t)\ {\rm d}x=-\frac{\partial u}{\n}(z_\ell,t),
\quad  \ell=1,2
$$
and completes the proof.
\end{proof}

Since $\psi_\ell^M\in L^2(\Omega),$ we can represent its Fourier expansion as 
$\psi_\ell^M=\sum_{n=1}^{\infty} a_{\ell,n}^M \varphi_n$.
This result, Lemma~\ref{mittag_derivative}, \eqref{solution} and
\cite[Theorem 3.1]{HettlichRundell:2001} lead to the following corollary. 
\begin{corollary}
The spectral representation of $w_\ell^M$ is 
\begin{equation}\label{eqn:spectral_rep}
w_\ell^M(x,t)=\sum_{n=1}^\infty a_{\ell,n}^M [1-E_{\alpha,1}(-\lambda_nt^\alpha)]\varphi_n(x),
\end{equation}
where 
\begin{equation}\label{alM}
a_{\ell,n}:=\lim_{M\to \infty}a_{\ell,n}^M=(w_n/\sqrt{\lambda_n})~J_{m+1}(\sqrt{\lambda_n})~ \xi_m^{\{c,s\}}(z_\ell).
\end{equation} 
\end{corollary}

\subsection{Uniqueness theorem}
\begin{theorem}\label{uniqueness}
Denote the solutions of \eqref{FDE} with respect to $D_1$ and $D_2$ by
$u_j,\ j=1,2,$ and
$z_1=(\cos{\theta_1},\sin{\theta_1}), z_2=(\cos{\theta_2},\sin{\theta_2})$
satisfy the condition
\begin{equation}\label{eqn:theta_cond}
\theta_1-\theta_2 \notin \pi \mathbb{Q}
\end{equation}
where $\mathbb{Q}$ is the set of rational numbers.
Then  
$$\frac{\partial u_1}{\n}(z_\ell, t)=\frac{\partial u_2}{\n}(z_\ell, t), \ t\in(0,T),\ \ell=1,2$$
implies that $D_1=D_2.$ 	
\end{theorem}
\begin{proof}
Without loss of generality we can let $\theta_1=0$.
By Lemma~\ref{wlM} and \eqref{eqn:spectral_rep}, we obtain 
\begin{equation}\label{eqn:spectral_cond}
  \sum_{n=1}^\infty a_{\ell,n}[1-E_{\alpha,1}(-\lambda_nt^\alpha)] 
  \left(\int_{D_1} \varphi_n(x)\ {\rm d}x -\int_{D_2} \varphi_n(x) 
  \ {\rm d}x\right)=0, \ t\in (0,T),\ \ell=1,2.
\end{equation}
The analyticity of the Mittag-Leffler function 
$E_{\alpha,1}(-\lambda_nt^\alpha)$ gives 
\begin{equation}\label{IML}
\sum_{n=1}^\infty a_{\ell,n}I_n[1-E_{\alpha,1}(-\lambda_nt^\alpha)]=0, 
\quad t\in (0, \infty),
\end{equation}
where
$$
I_n:=\int_{D_1} \varphi_n(x)\ {\rm d}x -\int_{D_2} \varphi_n(x) \ {\rm d}x. 
$$
Denoting the distinct eigenvalues of the Laplacian again by 
$\{\lambda_k:k\in \mathbb{N}^+\}$ and taking the Laplace transform
$t\to s$ in \eqref{IML}, we have
\begin{equation*}
 \sum_{k=1}^\infty \left(\sum_{\lambda_n=\lambda_k} a_{\ell,n} 
 I_n\right)\frac{\lambda_k}{(s^\alpha+\lambda_k)}=0, \quad s\in \mathbb{C}.
\end{equation*}
Letting $\eta=s^\alpha$  shows that the function
\begin{equation}\label{equality_4}
\Xi(\eta) := \sum_{k=1}^\infty \left(\sum_{\lambda_n=\lambda_k} a_{\ell,n} 
 I_n\right)\frac{\lambda_k}{\eta+\lambda_k}=0
\end{equation}
is analytic in $\eta$ with poles at $\eta=\{-\lambda_k\}$ and corresponding
residues $\{\lambda_k \sum_{\lambda_n=\lambda_k} a_{\ell,n}I_n\}_k$.
However, since $\Xi(\eta)$ vanishes identically for $\eta$ real and positive,
it follows that these residues must be zero.
Then by the strict positivity of $\lambda_k$ we see that  
$\sum_{\lambda_n=\lambda_k} a_{\ell,n} I_n=0$ for $\ell=1,2$ and
each eigenvalue $\lambda_k$ of the Laplacian.

\par
For a fixed eigenvalue $\lambda_k$, denote its corresponding eigenfunctions by
$\varphi_{n_k}$ and $\varphi_{n_k+1}$.
These have different phases and hence
\begin{equation}\label{equality_3}
\sum_{n=n_k,n_k+1} a_{\ell,n} I_n=0,\quad \ell=1,2.
\end{equation} 
For the case of $\phi_{n_k}=0$,
since $\theta_1=0,\ \theta_1-\theta_2\notin\pi\mathbb{Q},$
then \eqref{alM} implies $a_{1,n_k}\ne 0,\ a_{1,n_k+1}=0$ and
$a_{2,n_k+1}\ne 0$.
Inserting this into \eqref{equality_3} yields $I_{n_k}=0$.
The above result means $a_{2,n_k+1} I_{n_k+1}=0$,
which together with $a_{2,n_k+1}\ne 0$ gives $I_{n_k+1}=0$.
Analogously, for the case of $\phi_{n_k}=\pi/2$, we can prove
$I_{n_k}=I_{n_k+1}=0$.
Hence, we can conclude that for each eigenvalue
$\lambda_k\in \{\lambda_n:\N +\}$, $I_{n_k}=I_{n_k+1}=0$, which means 
$$
\int_{D_1} \varphi_n(x)\ {\rm d}x -\int_{D_2} \varphi_n(x) \ {\rm d}x
=\int_{\Omega}(\chi_{{}_{D_1}}-\chi_{{}_{D_2}})\varphi_n(x)\ {\rm d}x =0,\ \N+. 
$$ 
This result, the completeness of $\{\varphi_n(x):\N +\}$ and the continuity
of the boundaries of $D_1$ and $D_2$ give that $D_1=D_2$.
\end{proof} 

\par
In practice, it is certainly possible that the measured data
can only be obtained after some initial time $T_0$ has elapsed, i.e. 
only $g_\ell(t),\ t\in [T_0,T]$ is obtained.
Hence, the following corollary is  important;
its proof follows immediately from the analyticity of the Mittag-Leffler
function and the proof of Theorem~\ref{uniqueness}. 

\begin{corollary}\label{cut}
With the same conditions of Theorem \ref{uniqueness} and a constant $T_0\in (0,T),$ $$\frac{\partial u_1}{\n}(z_\ell, t)=\frac{\partial u_2}{\n}(z_\ell, t)\ \text{on}\ [T_0,T],
\ \ell=1,2$$ will also imply $D_1=D_2$.
\end{corollary}

\begin{remark}
The condition $\theta_1-\theta_2\notin \pi\mathbb{Q}$ is almost impossible
to satisfy in practice.
However, as we will show,  in the numerical section,
we only use the partial sum of the solution series to approximate the
exact spectral representation.
By taking a truncated basis, that is spectral cut-off of the functions
used to represent $\partial D$, we can show that satisfying
\eqref{eqn:theta_cond} is feasible.
Since in this case the number of eigenvalues is finite,
the upper bound $M$ of the degrees for the corresponding Bessel function
will also be finite.
Hence, in numerical reconstructions the condition
$\theta_1-\theta_2\notin \pi\mathbb{Q}$ can be weakened to
\begin{equation}\label{eqn:m-theta_cond}
\sin{m(\theta_1-\theta_2)}\ne 0,\quad m=1,2,\dots, M.
\end{equation}
\end{remark}

\subsection{The operators $G$ and $G'$}
In order to use Newton's method to recover $D$, we need to construct the
operator $G$ which maps $D$ to the flux data $\frac{\partial u}{\n}(z_\ell,t)$
then compute and demonstrate needed properties of its derivative $G'$.
In particular, to show the injectivity of $G'$.  

Recall that we have assumed the boundary of $D$ is star-like, i.e.   
$$
\partial D=\{q(\theta)(\cos{\theta},\sin{\theta})^{\top}:\theta\in[0,2\pi]\}.$$
Then by \eqref{solution}, the representation of $u(x,t)$ will be 
\begin{equation}\label{u}
\begin{aligned}
u(x,t)&=\sum_{n=1}^{\infty}\left(\int_0^t\int_{\Omega} \chi_{{}_D}\varphi_n(y)(t-\tau)^{\alpha-1}
E_{\alpha,\alpha}(-\lambda_n(t-\tau)^\alpha)\,{\rm d}y\,{\rm d}\tau \right)\varphi_n(x) \\
&=\sum_{n=1}^{\infty} \lambda_n^{-1}(1-E_{\alpha,1}(-\lambda_nt^\alpha))\varphi_n(r,\theta)
\int_0^{2\pi}\!\!\int_0^{q(s)} \varphi_n(\rho,s)\rho\,{\rm d}\rho\,{\rm d}s.
\end{aligned}
\end{equation}	
Now we can define the operator $G$ as  $G:q\mapsto(\partial_r
u(1,\theta_1,t),\partial_r
u(1,\theta_2,t)),$ where $\theta_\ell,\ \ell=1,2$ are the polar angles of the observation points $z_\ell$ on $\partial\Omega$.
the polar representation of $\varphi_n$ is
$\varphi_n(r,\theta) = w_n J_m(\sqrt{\lambda_n}\,r)\cos(m\theta+\phi_n)$
and we use the relation $J'_m(z)=-J_{m+1}(z)+\frac{m}{z}J_m(z)$ and the fact
that $\sqrt{\lambda_n}$ is a zero of the m-th Bessel function $J_m$ to see
that the radial derivative  of the radial part of $\varphi_n$ is
$w_n\sqrt{\lambda_n}J_{m+1}(\sqrt{\lambda_n}).$
Thus a direct calculation from \eqref{u} yields the $\ell$-th component of $G$ as 
\begin{equation}\label{equality_3}
\begin{aligned}
G_\ell(q)(t)=\sum_{n=1}^{\infty}b_n[1-E_{\alpha,1}(-\lambda_nt^\alpha)]\cos{(m\theta_\ell\!-\!\phi_n)}
\int_0^{2\pi}\!\Phi_n(q(s))\cos{(ms-\phi_n)}\,{\rm d}s,
\end{aligned}
\end{equation}
where
$$
b_n=-w_n^2\lambda_n^{-3/2}J_{m+1}(\sqrt{\lambda_n}),\quad
\Phi_n(x):=\int_0^{x\sqrt{\lambda_n}}\!\!\rho\, J_m(\rho)\,{\rm d}\rho.
$$
To compute $w_n$ we require the integral
$\int_0^{2\pi}\int_0^1 \rho J_m(\sqrt{\lambda_n}\rho)^2d\rho$.
The recursion formulae $[t^{-m}J_m(t)]' = -t^{-m}J_{m+1}(t)$ and
$[t^{m}J_m(t)]' = t^{m}J_{m-1}(t)$ give the relations
$2tJ_m(t)^2 = [t^2 J_m(t)^2 - J_{m+1}J_{m-1}]'$ and 
$J_{m-1}(t) = J_m'(t) = -J_{m+1}(t)$.
These and the fact that $J_m(\sqrt{\lambda_n}) = 0$ show that
$\int_0^{2\pi}\int_0^1 \rho J_m(\sqrt{\lambda_n}\rho)^2d\rho =
\frac{1}{2}J_{m+1}(\sqrt{\lambda_n}\rho)^2$.
Thus $\|\phi_n\|_2^2 = 1/w_n^2 =
\frac{1}{2} \eta_n\pi J_{m+1}(\sqrt{\lambda_n}\rho)^2$ where
$\eta_n = 1$ if $m(n)=0$ and $\frac{1}{2}$ if $m>0$.
Combining all of these shows that
$$
b_n = \frac{1}{\eta_n \pi \lambda_n^{3/2} J_{m+1}(\sqrt{\lambda_n})}.
$$
These computations mirror those of \cite{HettlichRundell:2001}
for the parabolic case.
From \eqref{equality_3}, with the notation $\sum'$ which indicates the index over distinct eigenvalues, we obtain 
\begin{equation}\label{G}
\begin{aligned}
G_\ell(q)(t)&=\sum_{n=1}^{\infty}{}^{'}\, b_n(1-E_{\alpha,1}(-\lambda_nt^\alpha))
\Big[\cos{(m\theta_\ell)}
\int_0^{2\pi}\Phi_n(q(s))\cos{(ms)}\,{\rm d}s\\
&\qquad\quad+\sin{(m\theta_\ell)}
\int_0^{2\pi}\Phi_n(q(s))\sin{(ms)}\,{\rm d}s\Big]\\
&=\sum_{n=1}^{\infty}{}^{'}\  b_n(1-E_{\alpha,1}(-\lambda_nt^\alpha))
\int_0^{2\pi}\Phi_n(q(s))\cos{(m(s-\theta_\ell))}\,{\rm d}s,
\end{aligned}
\end{equation}   
and 
\begin{equation}\label{G'}
\begin{aligned}
G_\ell'[q]h(t)=\sum_{n=1}^{\infty}{}^{'}\,\lambda_nb_n(1-E_{\alpha,1}(-\lambda_nt^\alpha))
\int_0^{2\pi}\!\!q(s)J_m(\sqrt{\lambda_n}q(s))\cos{(m(s-\theta_\ell))}h(s)\,{\rm d}s.
\end{aligned}
\end{equation}
We can now define $G$ and $G'$ by

\begin{definition}
	$$
	G(q)(t)=\begin{bmatrix}
	G_1(q)(t) \vspace{5pt}\\	
	G_2(q)(t) 
	\end{bmatrix},
	\quad	
	G'[q]h(t)=\begin{bmatrix}
	G'_1[q]h(t) \vspace{5pt}\\
	G'_2[q]h(t) 
	\end{bmatrix},
	$$
where $G_\ell,\ G'_\ell,\ \ell=1,2$ are defined in \eqref{G} and \eqref{G'}.
\end{definition}

\subsection{Injectivity of $G'$} 
We are now able to show the injectivity of $G'$.  
\begin{corollary}
Under the condition \eqref{eqn:theta_cond}, $G'[q]h(t)=0$ implies that $h=0$.
\end{corollary}	
\begin{proof}
$G'[q]h(t)=0$ leads to  $G_1'[q]h(t)=G_2'[q]h(t)=0.$ 
Following the proof of Theorem~\ref{uniqueness}, we have 
$$
\int_0^{2\pi}q(s)J_m(\sqrt{\lambda_n}q(s))\cos{(m(s-\theta_\ell))}h(s)\ {\rm d}s=0,\  \N+,\ \ell=1,2.
$$ 
Applying the proof in \cite[Section 4]{HettlichRundell:2001}
shows that $h=0$.
\end{proof}

In the introduction we noted that nonhomogeneous initial/boundary conditions
can be added to \eqref{FDE} with no change in scope and the same holds true
if the source is of the form $a(t)\chi(D)$ for $a(t)$ known.
An interesting question arises if the time dependent $a(t)$ 
has to be determined as well as $D$.
Even in the case $a(t)$ is constant more than two observation points would
now be needed, but it is easy to see that three would suffice.
It is a reasonable conjecture that three points would also suffice
to determine in addition $a(t)$ although this isn't immediately clear.
Although the unknown source would still give rise to a linear fractional
equation with the advantage that representation results would still be clear,
the fact that the two unknowns $a(t)$ and $D$ are coupled in a nonlinear
fashion would add considerable complexity to the new operators $G$ and $G'$.

\section{Numerical reconstruction}

\subsection{Iterative algorithm} 
\par
In this section, Newton's method will be used to recover $q(\theta)$.
Due to the ill-posedness of this problem, regularization is necessary and 
we will use a combination of a prior assumption on $q(s)$ together with
Tikhonov's method which leads to the Levenberg-Marquardt-type formula 
\begin{equation}\label{iteration}
q_{n+1}=q_n+[(G'(q_n))^*~G'(q_n)+\beta P]^{-1}
(G'(q_n))^* (g^\delta-G(q_n)).
\end{equation} 
Here, $g^\delta$ denotes the perturbed  measured data with
$\|(g-g^\delta)/g\|_{C(0,T)}\le \delta$, 
$q_n$ is the n-th approximation of the radial term of the star-like boundary, 
$\beta$ is the regularized parameter and $P$ is the penalized matrix. 
In this section, we only consider the unknown $q$ to be taken from
the trigonometric polynomial space with dimension up to degree $M$, i.e. 
$$
q(\theta)=\frac{1}{2}q_0+\sum_{n=1}^{M}\left(q_n^c \cos{n\theta}+q_n^s\sin{n\theta}\right).
$$
As will be seen, the effective value for $M$ that can be obtained will
be quite small.
This itself provides a regularization by spectral cut off,
but if used alone it leads to a quite limited regularization possibility;
hence the combination with \eqref{iteration}.

We also want to ensure the approximated $q_n$ is sufficiently smooth and so
we set the penalty term be the $H^2$ semi-norm of $q_n,$ which implies that $P$ is a $(2M+1)\times(2M+1)$ diagonal matrix with 
$$P_{1,1}=1,\ P_{i+1,i+1}=P_{i+M+1,i+M+1}=i^2,\ i=1,\dots,M.$$ 
The stopping criterion used was
$\|g^\delta-G(q_n)\|_{L^2(0,T)}\le \epsilon,\ \epsilon=O(\delta)$.
A good initial approximation is often essential for the convergence
of Newton schemes in such interior domain reconstructions and the current case
is no different.
Fortunately, we have a simple method of achieving this as noted in
\cite{HettlichRundell:2001}.
We take $q_0$ to be a circle of radius $\bar r$ with centre
$\bar x = (\bar x_1,\bar x_2)$.
An extended circular source has exactly the same boundary effect as
a delta-function point source at its centre.
Such a pole would generate a disturbance equal to $G_\alpha(\bar x - z,t)$
where $G_\alpha$ is the fundamental solution for the subdiffusion operator 
in \eqref{FDE}.  
This solution is available as a Wright function,
$G_\alpha(x,t) = t^{-\alpha/2}M(|x|/t^{\alpha/2})$ where
$M(z) = \sum_0^\infty \frac{(-z)^n}{n!\Gamma(1- \frac{\alpha}{2}(n+1))}$, 
see~\cite{MainardiMuraPagnini:2010}.
However, we do not require such precision for the initial approximation
purpose.
We can take the time-independent version by approximation of the steady state
values for each flux $g_\ell(t_\infty)$.
This gives $m$ values at positions $z_\ell$ and we simply perform a
least-squares fit to obtain the centre $\bar x$ and weight $\bar\rho$
of the pole based on Laplace equation for a circle.
Then, since $\bar\rho = \pi \bar r^2$,
we readily obtain our approximating circle.
In the case of only two observation points there is insufficient information
in general and then we simply assume the approximating circle has centre the
origin.

\subsection{Decomposition of $G$ and $G'$}
From the definitions of $G$ and $G'$ we can see the convergence rates of
their series representations should be slow since the time-dependent term 
$1-E_{\alpha,1}(-\lambda_n t^\alpha)$ does not converge to zero for $n$ large.
Hence, we split $G,\ G'$ into their steady states and transient components as 
\begin{equation*}
\begin{aligned}
G_\ell(q)(t)=&\frac{\partial v }{\n}(z_\ell)-\sum_{n=1}^{\infty}{}^{'}\  b_nE_{\alpha,1}(-\lambda_nt^\alpha)
\int_0^{2\pi}\Phi_n(q(s))\cos{(m(s-\theta_\ell))}\ {\rm d}s,\\
G'_\ell(q)(t)=&\frac{\partial}{\partial q}\left(\frac{\partial v }{\n}(z_\ell)\right)\\
&-\sum_{n=1}^{\infty}{}^{'}\ \lambda_nb_nE_{\alpha,1}(-\lambda_nt^\alpha)
\int_0^{2\pi}q(s)J_m(\sqrt{\lambda_n}q(s))\cos{(m(s-\theta_\ell))}h(s)\ {\rm d}s, 
\end{aligned}
\end{equation*}
where $v$ is the solution of the equation  
\begin{equation*}
\begin{cases}
\begin{aligned}
-\triangle v(x)&=\chi_{{}_D}, &&x\in \Omega;\\
v(x)&=0, &&x\in \partial\Omega. 
\end{aligned}
\end{cases}
\end{equation*}
From 
\cite{HettlichRundell:1996} we can obtain
$\frac{\partial v }{\n}(z_\ell)$ and 
$\frac{\partial}{\partial q}\left(\frac{\partial v }{\n}(z_\ell)\right)$
from the following Fourier expansions 
\begin{equation*}
\begin{aligned}
\frac{\partial v }{\n}(z_\ell)&=\frac{a_0}{2}+\sum_{n=1}^{\infty}
\left(a_n^c\cos{n\theta_\ell}+a_n^s\sin{n\theta_\ell}\right),\\
\frac{\partial}{\partial q}\left(\frac{\partial v }{\n}(z_\ell)\right)
&=\frac{b_0}{2}+\sum_{n=1}^{\infty}
\left(b_n^c\cos{n\theta_\ell}+b_n^s\sin{n\theta_\ell}\right),
\end{aligned}
\end{equation*}
where 
\begin{equation*}
\begin{aligned}
a_n^c&=\frac{1}{(n+2)\pi}\int_0^{2\pi} 
[q(\theta)]^{n+2}\cos{n\theta}\, {\rm d}\theta,\qquad
&&a_n^s=\frac{1}{(n+2)\pi}\int_0^{2\pi} 
[q(\theta)]^{n+2}\sin{n\theta}\,{\rm d}\theta,\\
b_n^c&=\frac{1}{\pi}\int_0^{2\pi} 
[q(\theta)]^{n+1}\cos{n\theta}\ {\rm d}\theta,\qquad
&&b_n^s=\frac{1}{\pi}\int_0^{2\pi} 
[q(\theta)]^{n+1}\sin{n\theta}\ {\rm d}\theta.\\
\end{aligned}
\end{equation*}

\subsection{Forward problem and $L^1$ time-stepping}
\par To obtain the measured data $g$ and also to compute the forward map
we need to solve the \eqref{FDE} numerically.
The spectral representation of the solution $u(x,t)$ gives insight to the
problem but as our forcing function is discontinuous, the convergence,
in particular that of the boundary derivative, is very slow.
This forces an extremely large number of eigenfunctions to be taken in order
to obtain sufficient accuracy.
As an alternative to the spectral representation
we use a finite difference representation in space and the
$L^1$ time-stepping method \cite{JinLazarovZhou:2016}
to discretize the fractional derivative $\D$
\begin{equation*}
\begin{aligned}
\D u(x,t_N) &= \frac{1}{\Gamma(1-\alpha)}\sum^{N-1}_{j=0}
\int^{t_{j+1}}_{t_j} \frac{\partial u(x,s)}{\partial s} 
(t_N-s)^{-\alpha}\, ds \\
&\approx \frac{1}{\Gamma(1-\alpha)}\sum^{N-1}_{j=0} 
\frac{u(x,t_{j+1})-u(x,t_j)}{\tau}\int_{t_j}^{t_{j+1}}
(t_N-s)^{-\alpha}ds\\
&=\sum_{j=0}^{N-1}b_j\frac{u(x,t_{N-j})-u(x,t_{N-j-1})}
{\tau^\alpha}\\
&=\tau^{-\alpha} [b_0u(x,t_N)-b_{N-1}u(x,t_0)
+\sum_{j=1}^{N-1}(b_j-b_{j-1})u(x,t_{N-j})] ,
\end{aligned}
\end{equation*}
where $\tau$ is the step size of the uniform partition on $t$ and 
\begin{equation*}
b_j=((j+1)^{1-\alpha}-j^{1-\alpha})/\Gamma(2-\alpha),\ j=0,1,\ldots,N-1.
\end{equation*}

\par For the Laplace operator $\triangle,$ the polar form 
$
\triangle u=\frac{\partial^2 u}{\partial r^2}+\frac{1}{r}\frac{\partial u}{\partial r}
+\frac{1}{r^2}\frac{\partial^2 u}{\partial \theta^2}
$
is used since the domain $\Omega$ is the unit disc in $\mathbb{R}^2.$
With uniformly partitions $\{r_l\},\ \{\theta_k\}$ on the radius $r\in(0,1)$ and the 
angle $\theta\in[0,2\pi)$ respectively, the discretized form of $-\triangle$ is 
\begin{equation*}
\begin{aligned}
-\triangle u(l,k,t_N)=&
-\frac{1}{h_r^2}[u(l+1,k,t_N)+u(l-1,k,t_N)-2u(l,k,t_N)]\\
&-\frac{1}{2lh_r^2}[u(l+1,k,t_N)-u(l-1,k,t_N)]\\ &-\frac{1}{l^2h_r^2h_\theta^2}[u(l,k+1,t_N)+u(l,k-1,t_N)-2u(l,k,t_N)]\\
=&(-\frac{1}{h_r^2}+\frac{1}{2lh_r^2})u(l-1,k,t_N)+(-\frac{1}{h_r^2}-\frac{1}{2lh_r^2})u(l+1,k,t_N)\\
&+(\frac{2}{h_r^2}+\frac{2}{l^2h_r^2h_\theta^2})u(l,k,t_N)-\frac{u(l,k+1,t_N)}{l^2h_r^2h_\theta^2}-\frac{u(l,k-1,t_N)}{l^2h_r^2h_\theta^2},
\end{aligned}
\end{equation*}
where $u(l,k,t_N)=u(r_l,\theta_k,t_N),$ and $h_r,\ h_\theta$ are the step sizes of the partitions on $r$ 
and $\theta$ respectively.
Hence, the finite difference scheme of the forward problem of \eqref{FDE} is  
\begin{equation*}
\begin{aligned}
&(\tau^{-\alpha}b_0+\frac{2}{h_r^2}+\frac{2}{l^2h_r^2h_\theta^2})u(l,k,t_N)+
(-\frac{1}{h_r^2}+\frac{1}{2lh_r^2})u(l-1,k,t_N)\\
&+(-\frac{1}{h_r^2}-\frac{1}{2lh_r^2})u(l+1,k,t_N)+(-\frac{1}{l^2h_r^2h_\theta^2})u(l,k+1,t_N)+(-\frac{1}{l^2h_r^2h_\theta^2})u(l,k-1,t_N)\\
&\qquad\qquad\qquad = \tau^{-\alpha} b_{N-1}u(l,k,t_0)
-\sum_{j=1}^{N-1}\tau^{-\alpha}(b_j-b_{j-1})u(l,k,t_{N-j})+\chi_{{}_D}(l,k).
\end{aligned}
\end{equation*}

\subsection{Numerical results}
The purpose of this section is to investigate our ability to perform
reconstructions and in particular to investigate the difference
as a function of $\alpha$.
We will also look at the effect of different placements of the measurements
points, of the noise level in the data.
This will be accomplished by a series of experiments to be outlined below.

In all the figures to be shown, the legend is the following:
the (blue) dotted line is the exact curve; the (red) dashed line is the reconstructed curve;
and the bulleted points on the (blue) solid circle representing
the exterior boundary $\partial\Omega$ are the observation points $z_\ell$.

We first take $\alpha=0.9,$ the final time $T=1,$
the regularized parameter $\beta=10^{-2}$.
We suppose the data $g_\ell(t)$ has uniform random added noise of $\delta$
times the value.
Then the following experiments were constructed. 
\begin{equation*} 
\begin{aligned}
& E_{1a}:\quad &&q(\theta)=0.6+0.1\cos{\theta}+0.1\sin{2\theta}, \quad \theta_1=\frac{15}{32}\pi,\    
\theta_2=\frac{19}{16}\pi,\ \ \epsilon=\delta/2; 
	\\	
& E_{1b}:\quad &&q(\theta)=0.6+0.1\cos{\theta}+0.1\sin{2\theta}, \quad
\theta_1=\frac{3}{4}\pi,\phantom{1}\  \theta_2=\frac{55}{32}\pi,\ \ 
\epsilon=\delta/2.
\end{aligned}
\end{equation*}
Experiments $E_{1a}$ and $E_{1b}$
have the same exact radius function $q(\theta)$.
However, the locations of observation points are different and this leads
to the difference between reconstructions of these two experiments.
See Figure~\ref{e1ab} for an illustration of the fact that
the reconstructed domain $D$ depends strongly on the location
of the observation points.

The left figure here is with $1\%$ noise, but actually
even a significant change in the noise level
(5\% against 1\%) has little bearing in this respect, the former being
only slightly worse.
The change of the observation points in $E_{1b}$ shown in the middle
and rightmost figures makes an enormous difference here;
reconstructions are considerably improved.

\begin{figure}[th!]\label{e1ab}
	\center
\begin{subfigure}
  \centering
  \includegraphics[trim = 1.3cm .1cm 1.3cm .1cm, clip=true,height=5.75cm,width=5.75cm]{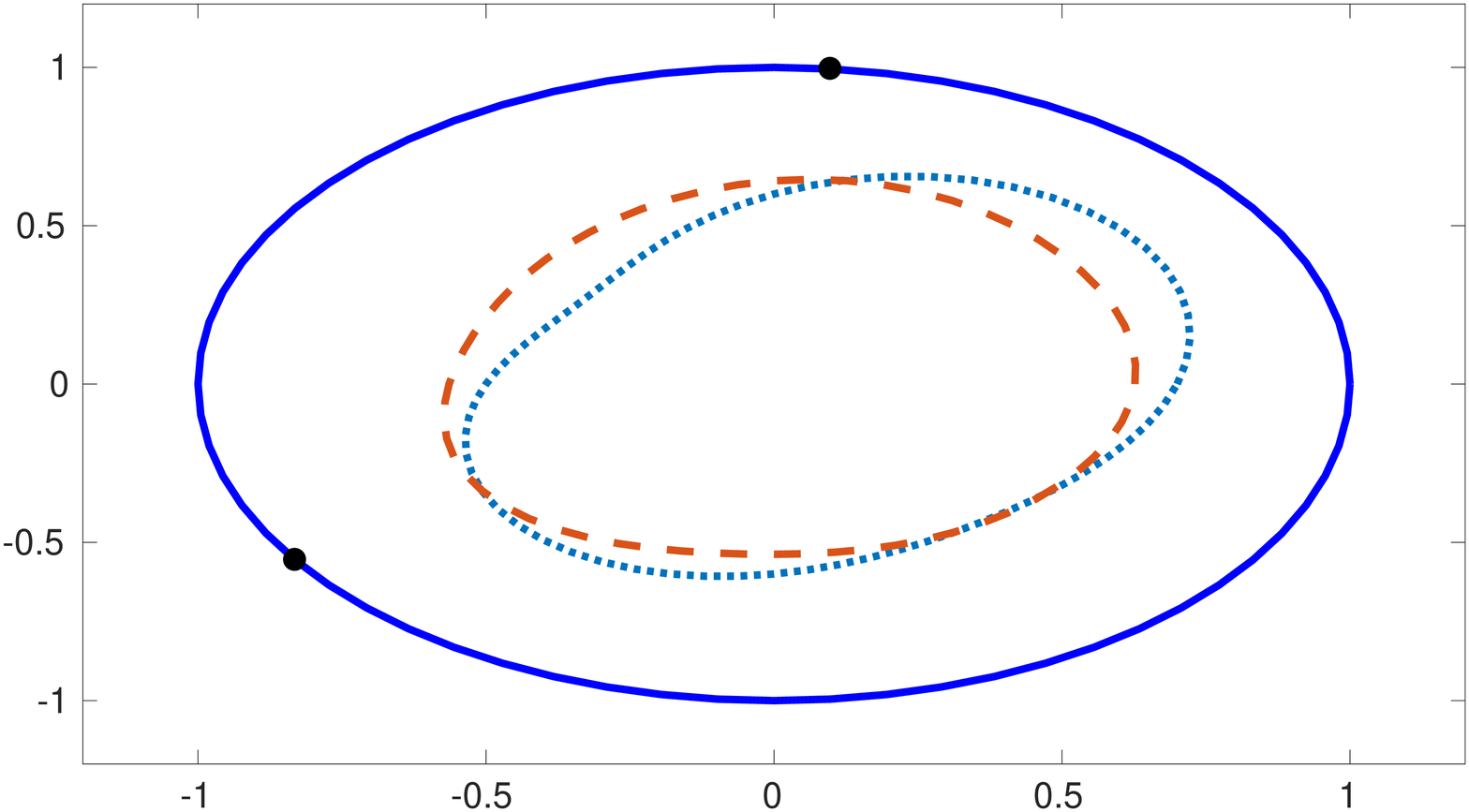}
\end{subfigure}
\begin{subfigure}
  \centering
  \includegraphics[trim = 1.1cm .1cm 1.3cm .1cm, clip=true,height=5.75cm,width=5.75cm] {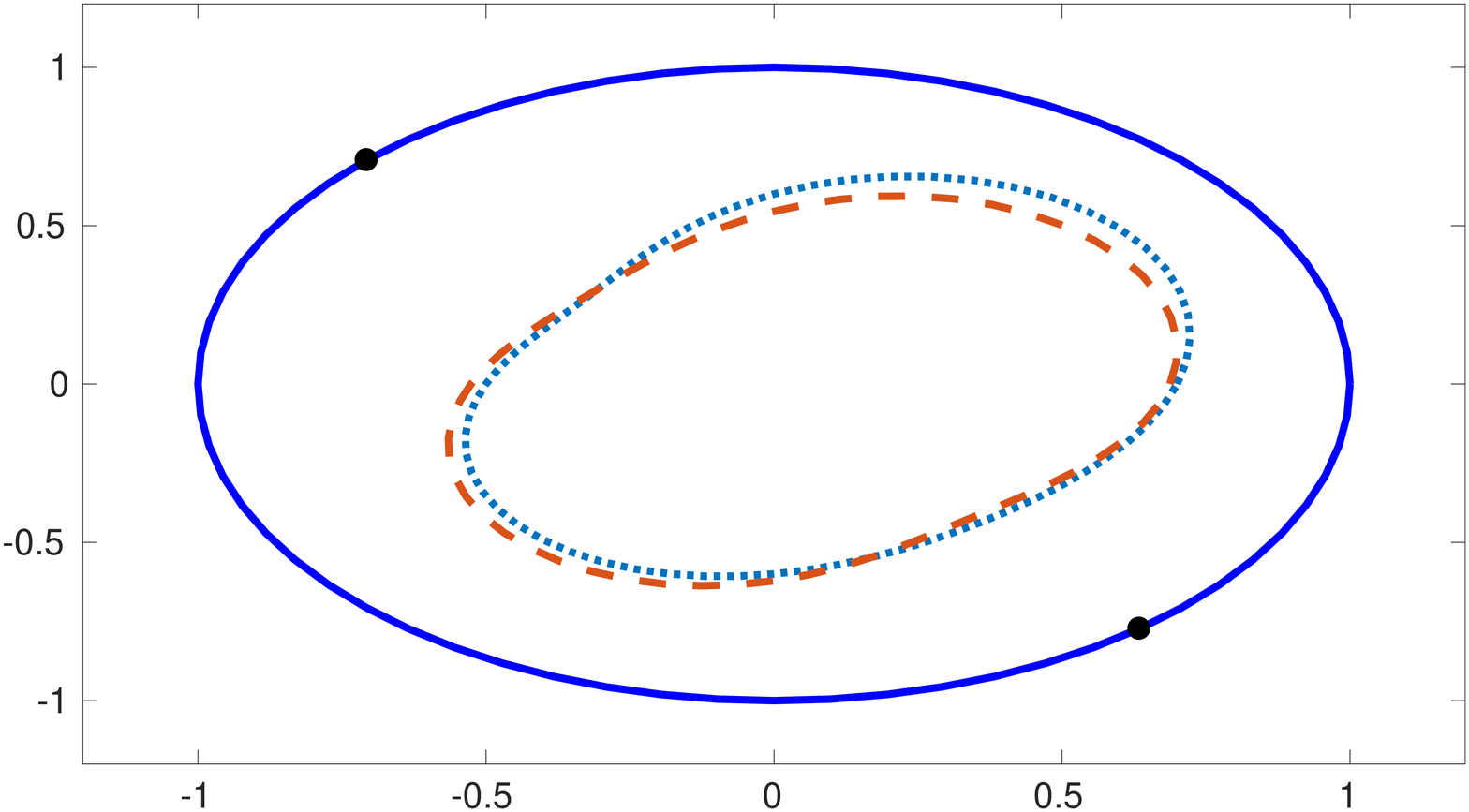}
\end{subfigure}
\begin{subfigure}
  \centering
  \includegraphics[trim = 1.1cm .1cm 1.3cm .1cm, clip=true,height=5.75cm,width=5.75cm] {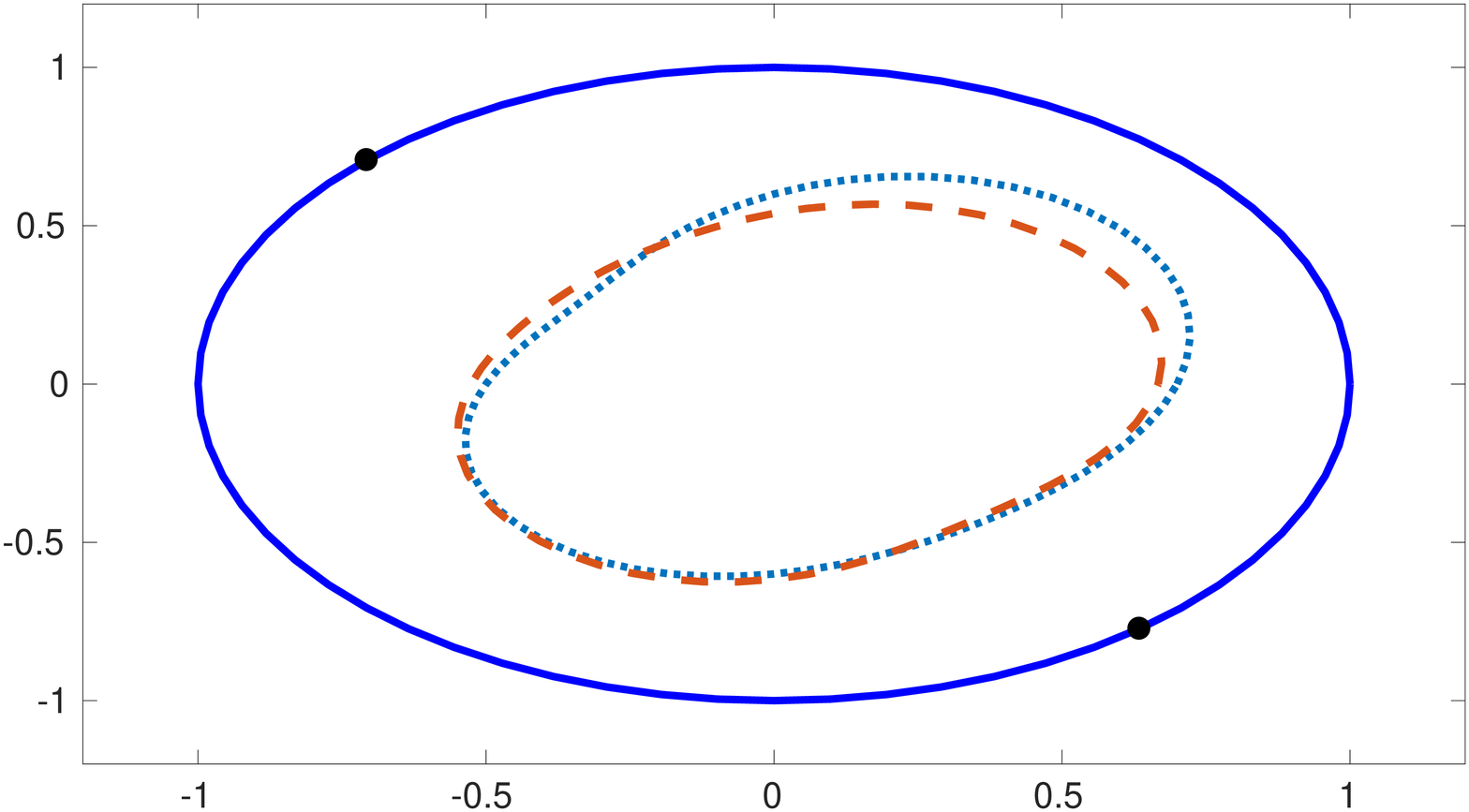}
\end{subfigure}
\vskip-20pt
\caption{\small Exact $q$ and numerical approximation\ \; ($\alpha=0.9$).
\\
Left: $E_{1a}$, $\delta=1\%$;\quad Middle:  $E_{1b}$, $\delta=1\%$;\quad
Right: $E_{1b}$, $\delta=5\%$.
}
\label{e1ab}
\end{figure}

This prompts us to redo this experiments to find the relation between
curve features and observation points in the reconstruction.
\begin{equation*} 
\begin{aligned}	
& E_{2a}:\quad &&q(\theta)=0.5+0.05\cos{\theta}+0.3\sin{2\theta},\quad
\theta_1=0,\quad \ \ 
\theta_2=\frac{31}{32}\pi,\ 
\epsilon=\delta/10; 
\\		
& E_{2b}:\quad &&q(\theta)=0.5+0.05\cos{\theta}+0.3\sin{2\theta},\quad
\theta_1=\frac{23}{32}\pi,\   
\theta_2=\frac{27}{16}\pi,\ 
\epsilon=\delta/10. 
\end{aligned}
\end{equation*}

The reconstruction pairs in Figure~\ref{e2ab} express the expected outcome;
both the proximity and alignment of the observation points are to the critical
features of the exact $q$, the better is the obtained approximation.

\begin{figure}[th!]\label{e2ab}
	\center
\begin{subfigure}
  \centering
\includegraphics[trim = .5cm .5cm .5cm .5cm, clip=true,height=5.5cm,width=5.75cm]
		{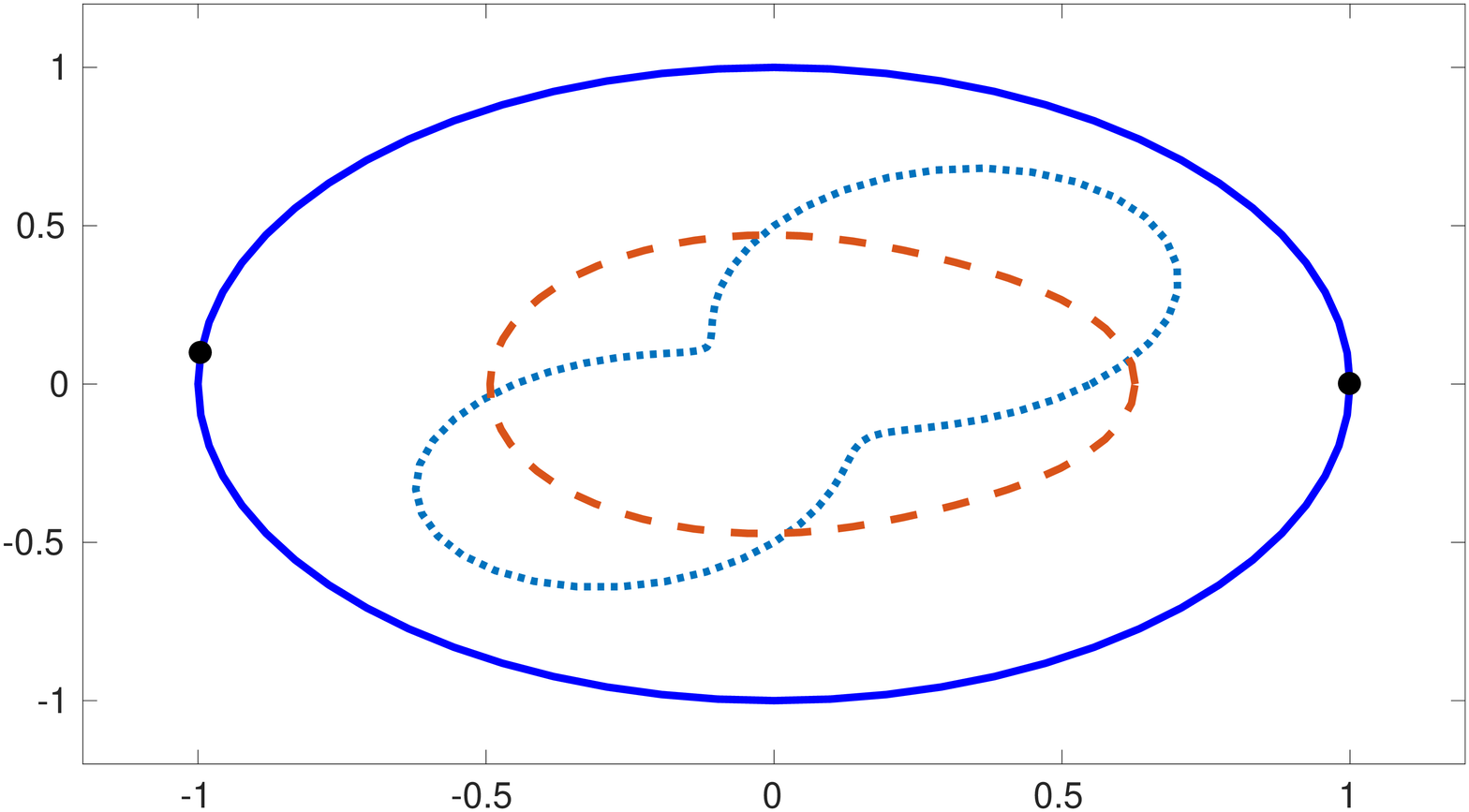}
\end{subfigure}
\begin{subfigure}
  \centering
\includegraphics[trim = .5cm .5cm .5cm .5cm, clip=true,height=5.5cm,width=5.75cm]
		{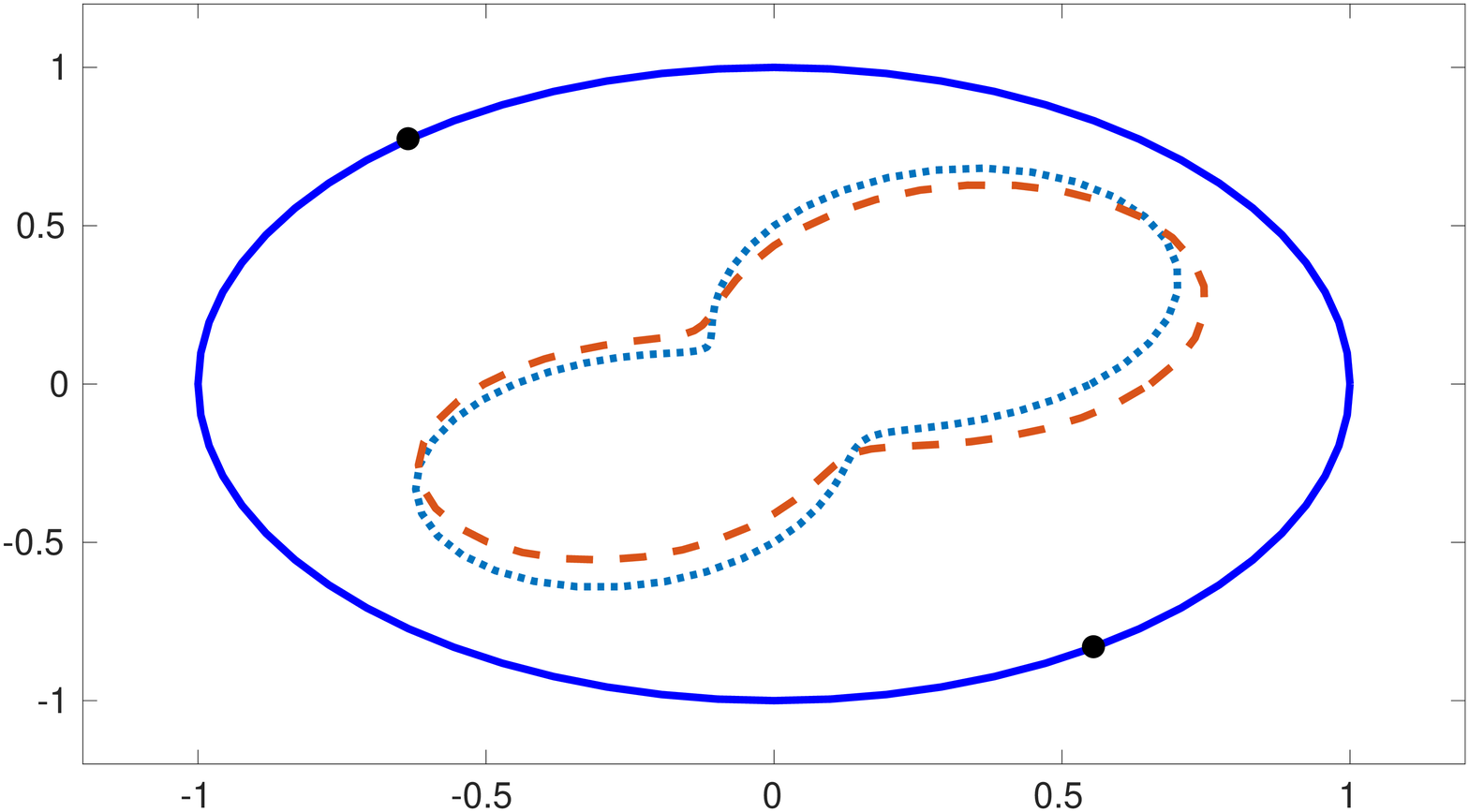}
\end{subfigure}
\vskip-15pt
\caption{\small Results of experiments $E_{2a}$ (left) and $E_{2b}$ (right),
\quad $\delta=1\%$, $\alpha=0.9$.}
\label{e2ab}
\end{figure}

A rigorous theoretical proof of this would be extremely useful but
the observation is widely reported in other situations.
For example, in inverse obstacle scattering there is a shadow region
on the reverse side of an incident wave from a given direction.
While all these problems do have strong diffusion and the theoretical
ability to ``wrap around'' obstacles, this is still limited.

\subsection{Fractional vs classical diffusion reconstructions}

An obvious question is how the reconstructions will depend on the
fractional diffusion parameter $\alpha$.
First we look at a profile of a typical data measurement $g(t)$ --
in this case for a circular inclusion with centre the origin.
\newbox\figinit
\newdimen\figlen  \figlen = 0.4\hsize

\setbox\figinit\vbox{\hsize=\figlen
{\footnotesize
\beginpicture
  \setcoordinatesystem units <0.5\figlen,5\figlen>
  \setplotarea x from 0 to 2.0, y from 0 to 0.2
  \axis bottom ticks numbered from 0 to 2 by 0.5 /
  \axis left ticks numbered from 0 to 0.2 by 0.05 /
\put {{\color{blue}$\alpha=1$}} [l] at 1.5 0.05
\put {{\color{red}$\alpha=1/2$}} [l] at 1.5 0.07
\linethickness=0.6pt
\put {{$E_{\alpha,1}(-\pi^2 t^\alpha)$}} [lt] at 0.02 2
\put {$t$} [rb] at 2.0 0.002
\setlinear
\setsolid
{\color{blue}\plot 
         0         0
    0.0200    0.0015
    0.0400    0.0030
    0.0600    0.0043
    0.0800    0.0055
    0.1000    0.0066
    0.1200    0.0076
    0.1400    0.0085
    0.1600    0.0093
    0.1800    0.0100
    0.2000    0.0106
    0.2200    0.0112
    0.2400    0.0117
    0.2600    0.0122
    0.2800    0.0126
    0.3000    0.0130
    0.3200    0.0133
    0.3400    0.0136
    0.3600    0.0139
    0.3800    0.0141
    0.4000    0.0143
    0.4200    0.0145
    0.4400    0.0147
    0.4600    0.0148
    0.4800    0.0150
    0.5000    0.0151
    0.5200    0.0152
    0.5400    0.0153
    0.5600    0.0154
    0.5800    0.0154
    0.6000    0.0155
    0.6200    0.0156
    0.6400    0.0156
    0.6600    0.0157
    0.6800    0.0157
    0.7000    0.0157
    0.7200    0.0158
    0.7400    0.0158
    0.7600    0.0158
    0.7800    0.0159
    0.8000    0.0159
    0.8200    0.0159
    0.8400    0.0159
    0.8600    0.0159
    0.8800    0.0159
    0.9000    0.0159
    0.9200    0.0160
    0.9500    0.0160
    1.0000    0.0160
    1.2500    0.0160
    1.5000    0.0160
    1.7500    0.0160
    2.1000    0.0160
/ }\relax
\setdots <2pt>
{\color{red}\plot
         0         0
    0.0200    0.0080
    0.0400    0.0095
    0.0600    0.0104
    0.0800    0.0110
    0.1000    0.0114
    0.1200    0.0118
    0.1400    0.0120
    0.1600    0.0122
    0.1800    0.0124
    0.2000    0.0126
    0.2200    0.0127
    0.2400    0.0129
    0.2600    0.0130
    0.2800    0.0131
    0.3000    0.0132
    0.3200    0.0133
    0.3400    0.0133
    0.3600    0.0134
    0.3800    0.0135
    0.4000    0.0135
    0.4200    0.0136
    0.4400    0.0136
    0.4600    0.0137
    0.4800    0.0137
    0.5000    0.0138
    0.5200    0.0138
    0.5400    0.0139
    0.5700    0.0139
    0.6000    0.0140
    0.6400    0.0140
    0.6800    0.0141
    0.7100    0.0141
    0.7500    0.0142
    0.7800    0.0142
    0.8000    0.0142
    0.8400    0.0143
    0.8500    0.0143
    0.9000    0.0143
    0.9500    0.0144
    1.0000    0.0144
    1.1000    0.0145
    1.2000    0.0146
    1.3000    0.0146
    1.4000    0.0147
    1.5000    0.0147
    1.6000    0.0148
    1.7000    0.0148
    1.8000    0.0148
    1.9000    0.0149
    2.0000    0.0149
/ }\relax
\endpicture
}
}

{ 
\begin{wrapfigure}{r}{0.35\textwidth}
  \begin{center}
\vskip-25pt
    \includegraphics[trim = 0.5cm 0.3cm .6cm .4cm, clip=true,height=5cm,width=6cm]{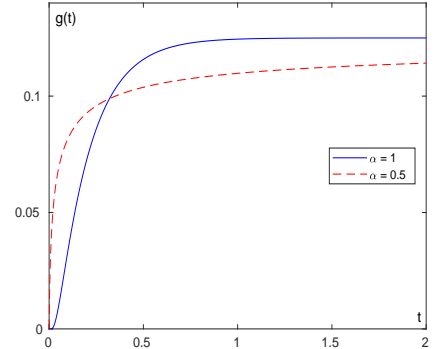}
  \end{center}
\vskip-15pt
  \caption{\small{ The data $g_\alpha(t)$.}}
\label{f0}
\end{wrapfigure}
Figure~\ref{f0} shows the function $g(t)$ for both $\alpha=1$
and $\alpha=\frac{1}{2}$.
In each case $g(t)$ goes to the same steady state value but how it
approaches is quite different.
In the case of the heat equation the effective steady state is reached
long before the endpoint chosen here of $T=2$.
Indeed, by $t=0.5$, 99\% of the steady state value has been achieved and is
typical of the behaviour expected by the exponential term in the solution
representation when $\alpha=1$.
When $\alpha=\frac{1}{2}$ the situation is quite different;
the Mittag-Leffler function decays only linearly for large (negative)
values of the argument and so steady state is achieved much more slowly.
In consequence, for $\alpha=1$ only time measurements made for small $t$
offer any utility in providing information, but for $\alpha<1$ this is
not the case.

}
The model \eqref{FDE} has the positivity property; the nonhomogeneous
forcing function and initial value are nonnegative and this implies the
solution $u(x,t)$ be nonnegative for all $(x,t)$,
see \cite{LiuRundellYamamoto:2016}.
Thus the (exact) overposed flux values consisting of the outer
normal derivative on $\partial\Omega$ will be negative for all $t$.
In fact these values must start at $0$ and monotonically decrease to the
steady state value predicted by the equation $-\triangle u = \chi(D)$
with the same Dirichlet condition on $\partial\Omega$ as imposed by
\eqref{FDE}.
From equation \eqref{eqn:spectral_rep} and the monotonicity of the
Mittag-Leffler function on the negative real axis the term
$\sigma_{\alpha,n}(t) := 1 - E_{\alpha,1}(-\lambda_n t^\alpha)$
is monotone and the range of this is within $[0,1)$ for all $t$.
Even if the time interval is truncated to $[0,T]$, since $\lambda_n\to\infty$
linearly in $n$, most of the modes will have the property that
$\sigma_{\alpha,n}(t)$ covers a substantial part of the range $(0,1]$.
However, this will not be independent of $\alpha$ as the growth of
$E_{\alpha,1}(-\lambda t^\alpha)$ depends on $\alpha$.
The larger the $\alpha$, the initially the slower, but finally the faster
the decay of $E_{\alpha,1}(-\lambda t^\alpha)$ to zero.
Thus, as we have seen in Figure~\ref{f0},
 the heat equation with $\alpha=1$ will reach steady state faster
than for $\alpha<1$ and the smaller the $\alpha$ the longer it will take to
reach steady state.
Of course the high frequency modes (large $\lambda_n$) will reach steady state
much faster and this is true for all $\alpha$.


{ 
\begin{wrapfigure}{r}{0.35\textwidth}
  \begin{center}
\vskip-25pt
    \includegraphics[trim = 0.5cm 0.4cm .6cm .3cm, clip=true,height=6cm,width=6cm]{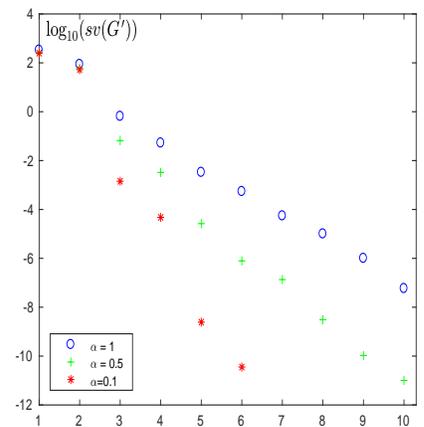}
  \end{center}
\vskip-15pt
  \caption{\small{ Singular values of $G'$.}}
\label{svalues}
\end{wrapfigure}
Figure~\ref{svalues} displays the singular values $\sigma_k$ of the operator
$(G')^*\circ G'$ for experiment $E_{2b}$. 
Note the obvious exponential decay of $\sigma_k$ for all $\alpha$.
This is to be expected due to the extreme ill-conditioning of the problem.
However, the rates do depend on $\alpha$; the smaller the $\alpha$
the greater the decay rate and hence degree of ill-conditioning.
Again, this must be expected as for small $\alpha$ the diffusion
is initially extremely rapid and the transient information cannot be
adequately captured.
Thus, while all cases require $g(t)$ for small values of $t$ this is
even more important the smaller the $\alpha$.
The slower growth of the profile $g(t)$ for larger $t$ cannot compensate.
Although this seems anomalous at first glance, the factor
$1- E_{\alpha,1}(-z)$ for large argument $z = \lambda_n t^\alpha$ approaches
unity with behaviour $\frac{c_1}{z} + \frac{c_2}{z^2} + \ldots$
where $c_k=c_k(\alpha)$.
Hence for modest values of $t$, say near $t=1$ but large $\lambda_n$
this is dominated by the first term with a rapidly diminishing contribution
to further terms $1/z^2$, $1/z^3\ \ldots$ and so also offers very little
information to be picked up from $g(t)$.

}
Note that while it is important to take a small step size initially in the
measurement of $g(t)$ this need not be continued for the entire interval.
Thus if we take say the first few measurements with $dt=0.001$ then
this can be steadily increased so that (say) over the last half of $[0,T]$ we
use a step size of $dt=0.1$; with this the reconstructions  differences
will be imperceptible.
In fact, the optimal measurement points $\{t_k\}$ should be chosen to
give approximately equal arc lengths of $u_r(1,\theta,t)=g(t)$.
This will mean a far greater concentration of point for small values of $t$
and this effect will be stronger the smaller the $\alpha$ value.

Reconstructions are shown for experiments $E_{1b}$ and $E_{2b}$ 
and for $\alpha = 0.1,\,0.5,\,1$ in Figure~\ref{e_diff_alpha}.
Here we took the initial step size in $t$ to be $dt=0.001$.

\begin{figure}[th!]\label{e_diff_alpha}
	\center
\begin{subfigure}
  \centering
\includegraphics[trim = .5cm .5cm .5cm .5cm, clip=true,height=5.5cm,width=5.75cm]{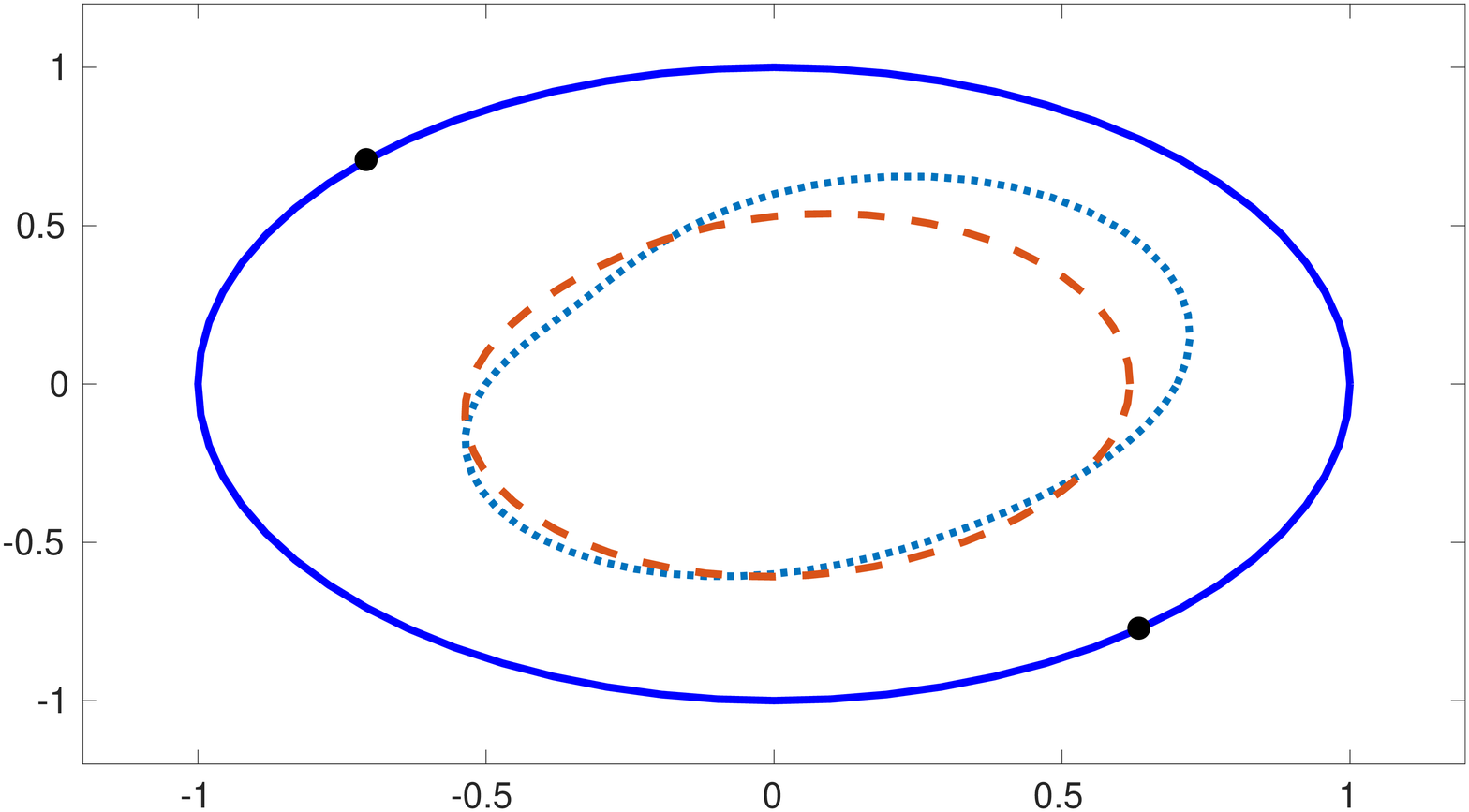}
\end{subfigure}
\begin{subfigure}
  \centering
\includegraphics[trim = .5cm .5cm .5cm .5cm, clip=true,height=5.5cm,width=5.75cm]{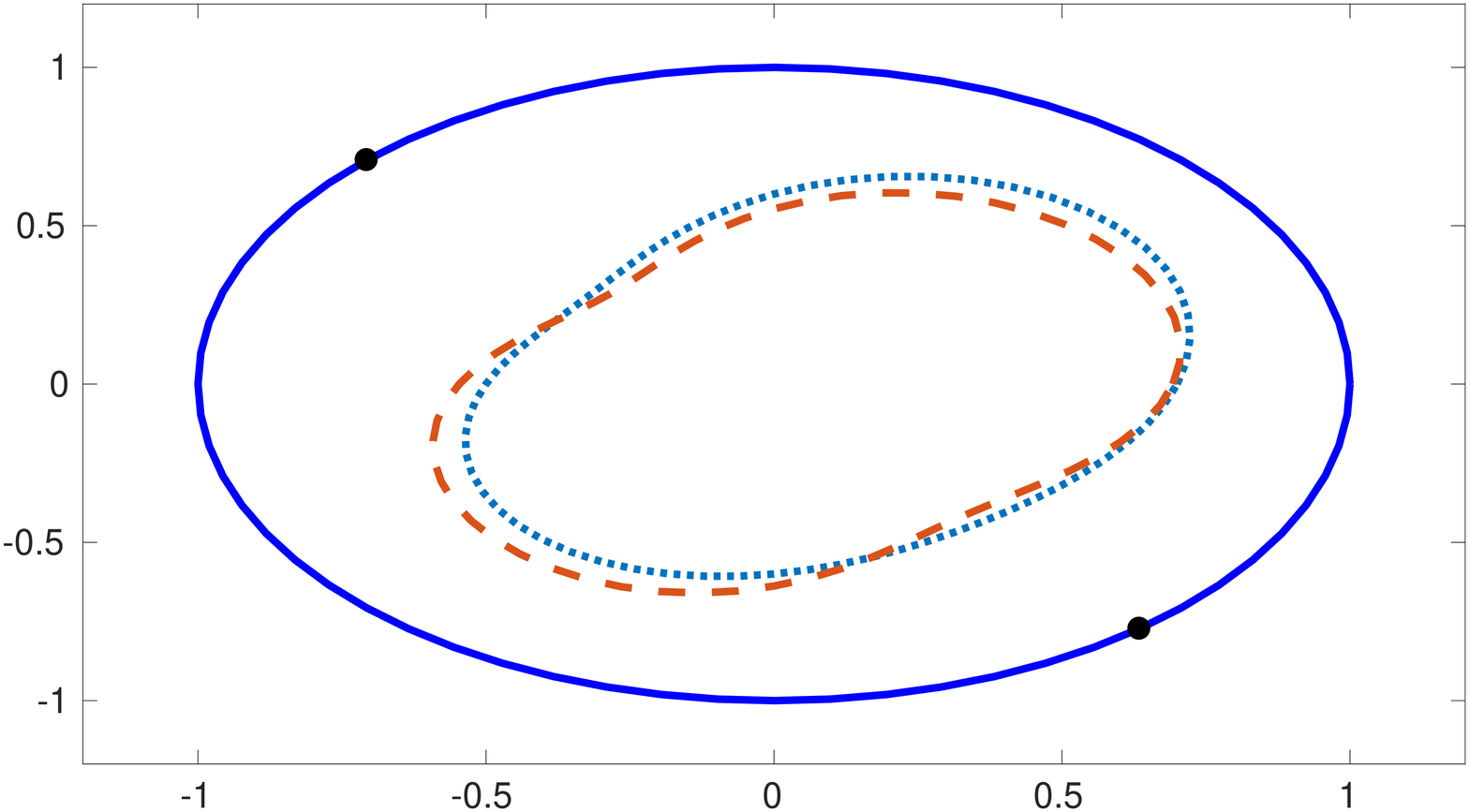}
\end{subfigure}
\begin{subfigure}
  \centering
\includegraphics[trim = .5cm .5cm .5cm .5cm, clip=true,height=5.5cm,width=5.75cm]{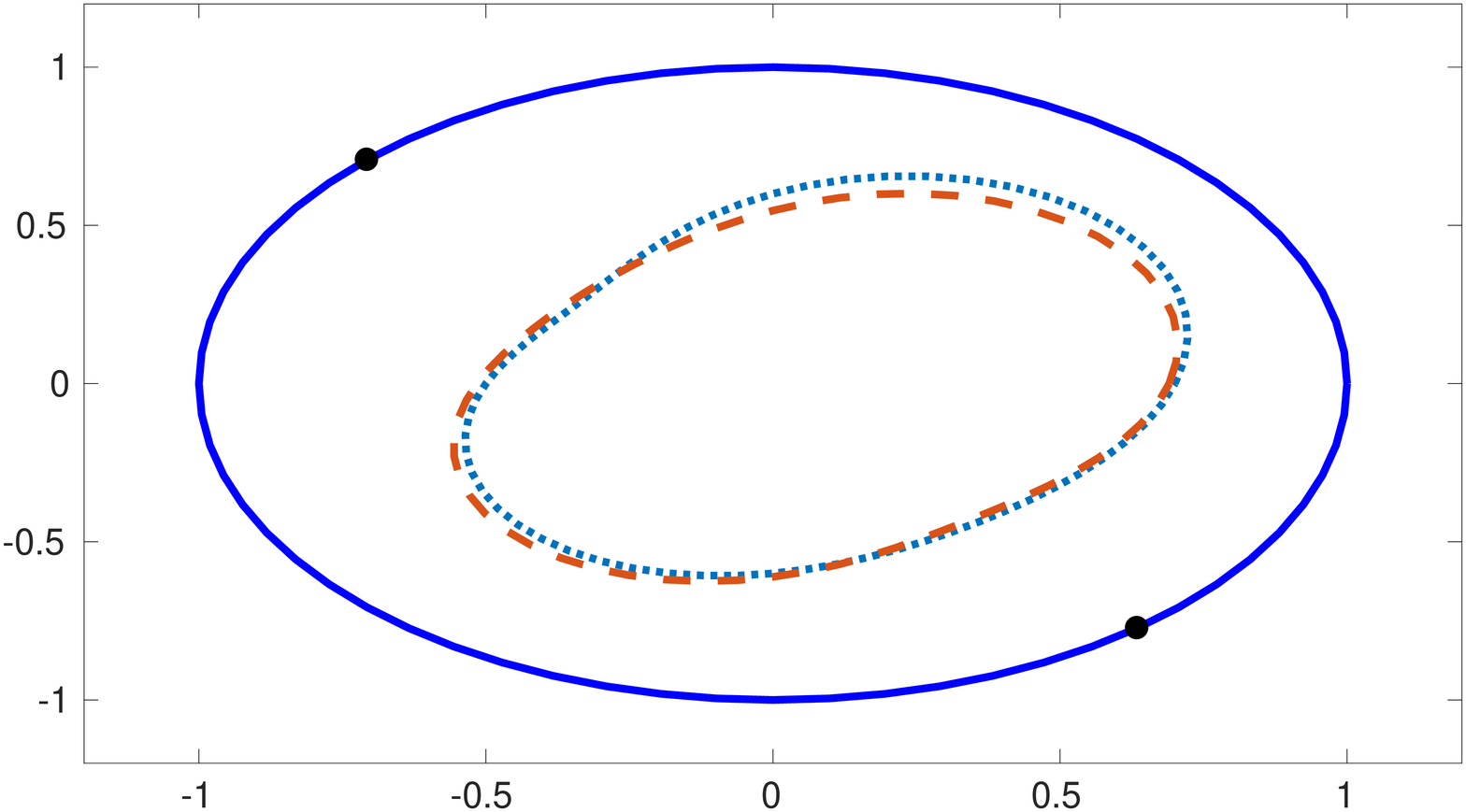}
\end{subfigure}
\vskip-15pt
\begin{subfigure}
  \centering
\includegraphics[trim = .5cm .5cm .5cm .5cm, clip=true,height=5.5cm,width=5.75cm] {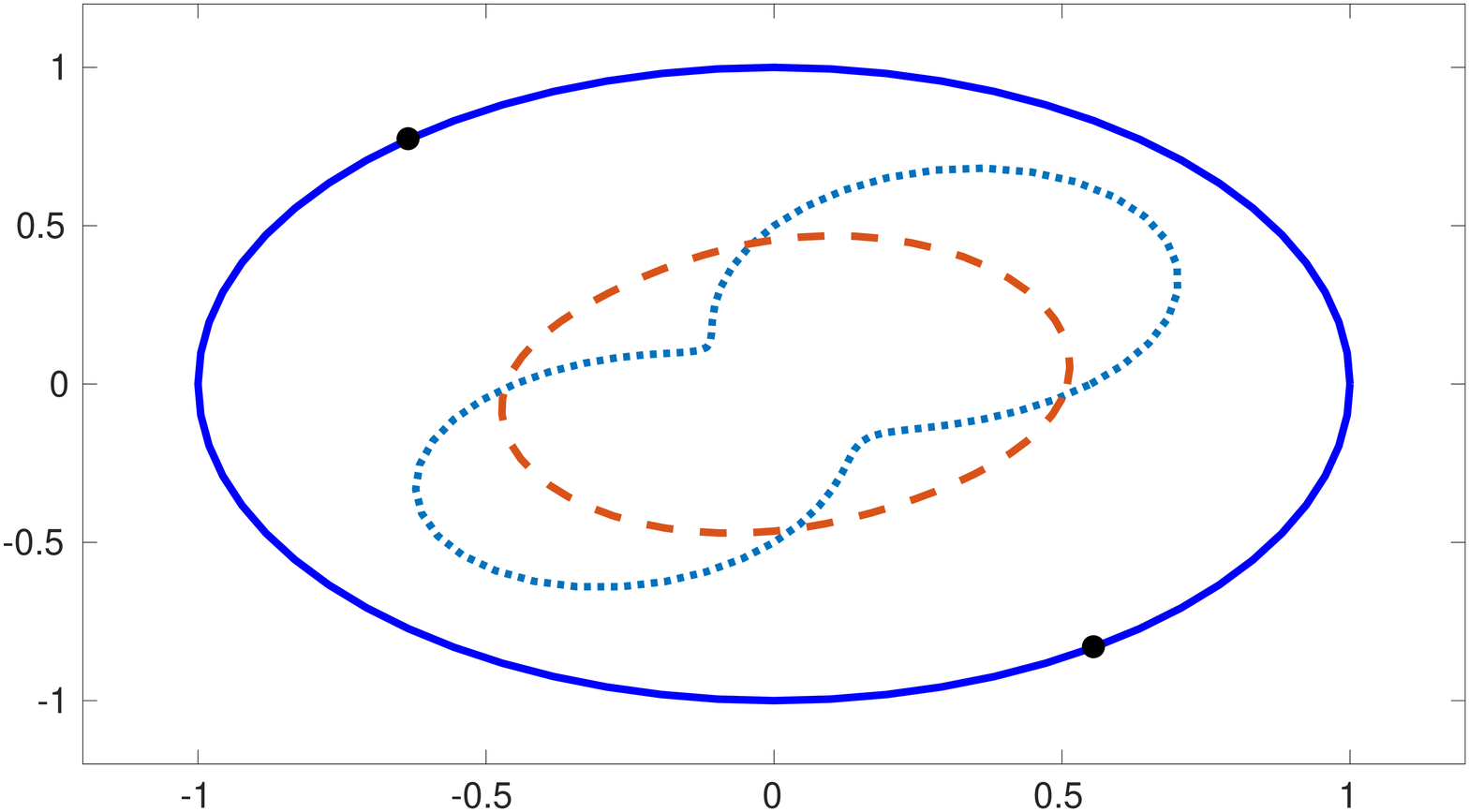}
\end{subfigure}
\begin{subfigure}
  \centering
\includegraphics[trim = .5cm .5cm .5cm .5cm, clip=true,height=5.5cm,width=5.75cm] {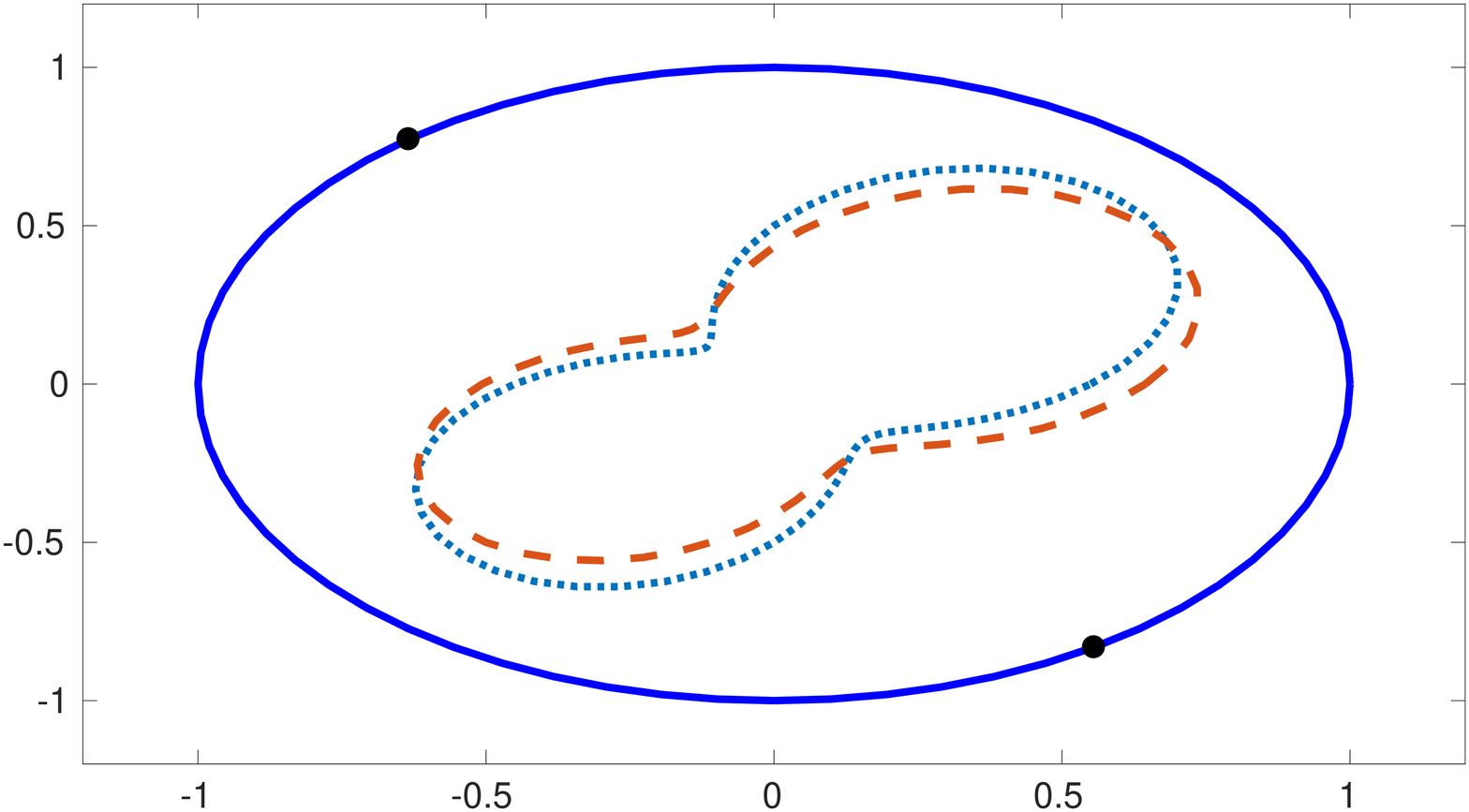}
\end{subfigure}
\begin{subfigure}
  \centering
\includegraphics[trim = .5cm .5cm .5cm .5cm, clip=true,height=5.5cm,width=5.75cm] {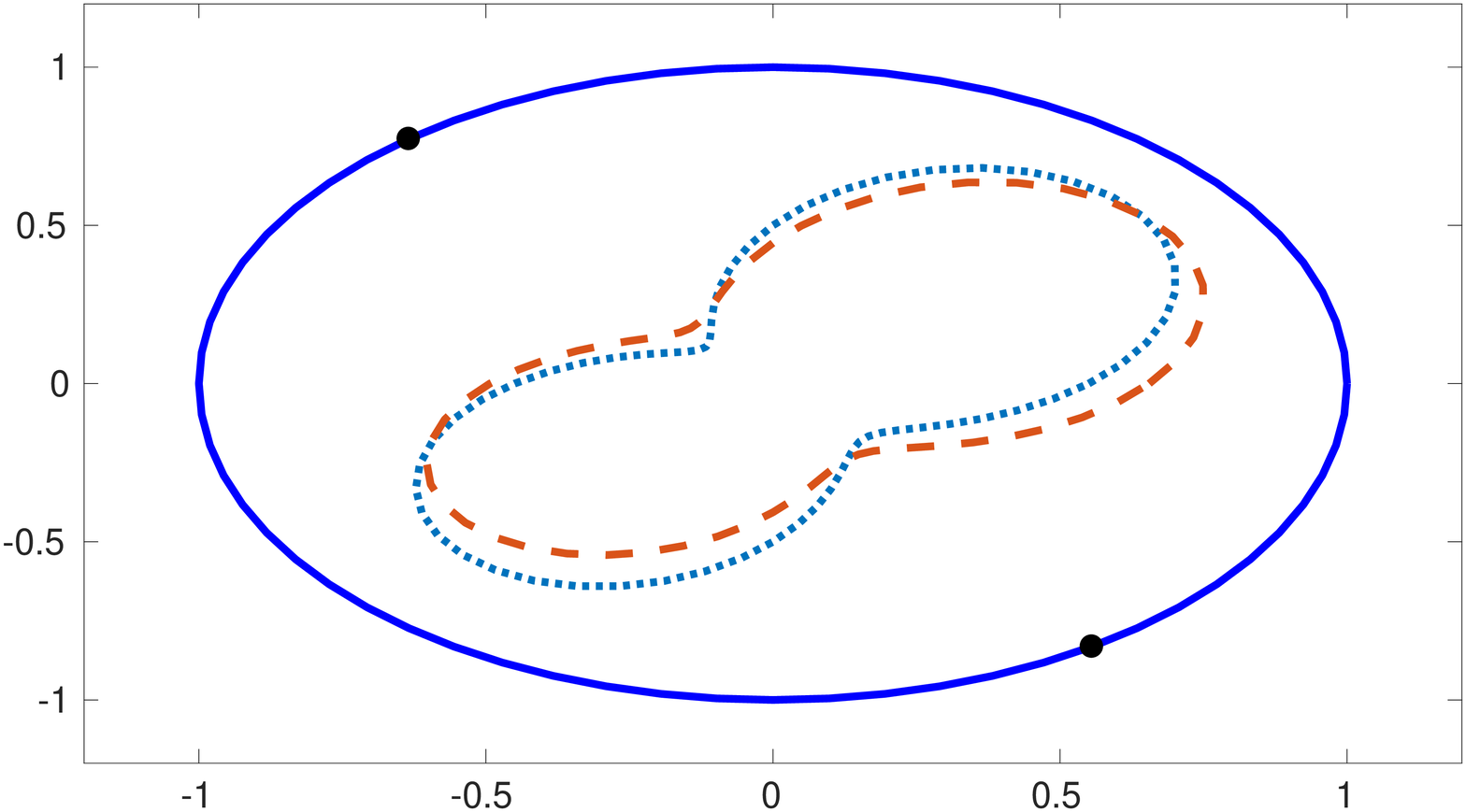}
\end{subfigure}
\vskip-20pt
\caption{\small $E_{1b}$ (top) and $E_{2b}$ (bottom)
for $\alpha = 0.1,\;0.5,\;1$.
  $\;\delta=1\%$.}
\label{e_diff_alpha}
\end{figure}

These bear out the previous observations and with Figure~\ref{svalues}.
The differences are relatively small for $\alpha$ close to $1$
but with a rapid deterioration, particularly in the higher frequency
information, with decreasing $\alpha$.
Thus in $E_{1b}$ the simple shaped-object has a virtually identical
reconstruction for $\alpha=1/2$ and $\alpha=1$ - both within the
variation expected with $1\%$, noise but the reconstruction is clearly
poorer for $\alpha=0.1$ where we are only able to determine the rough size 
and placement.
The similarity in $E_{2b}$ is due to the small initial time steps taken;
if instead we had to increase
$dt$ to $dt=0.01$ initially, then the difference between $\alpha=1/2$ and 
$\alpha=1$ would be much more evident.

What if we delay the flux measurements until a later time,
that is we measure only over $[T_0,T]$ for some $T_0>0$?
There are certainly physical situations where this might be required.
Note that Corollary~\ref{cut} indicates uniqueness will still hold
but the question is the resulting change in condition number.
In Figures~\ref{fig_cut}--\ref{e2b_cut_right}
we measure the flux data $g_\ell(t)$
over incomplete intervals.

Figure~\ref{fig_cut} shows the expected outcome; a decrease in the ability
to construct higher modes as short-time information is lost.
\begin{figure}[th!]\label{fig_cut}
	\center
\begin{subfigure}
  \centering
\includegraphics[trim = .8cm .9cm .3cm .3cm, clip=true,height=5.5cm,width=5.75cm]{e2b.eps}
\end{subfigure}
\begin{subfigure}
  \centering
\includegraphics[trim = .7cm .8cm .2cm .3cm, clip=true,height=5.5cm,width=5.75cm]{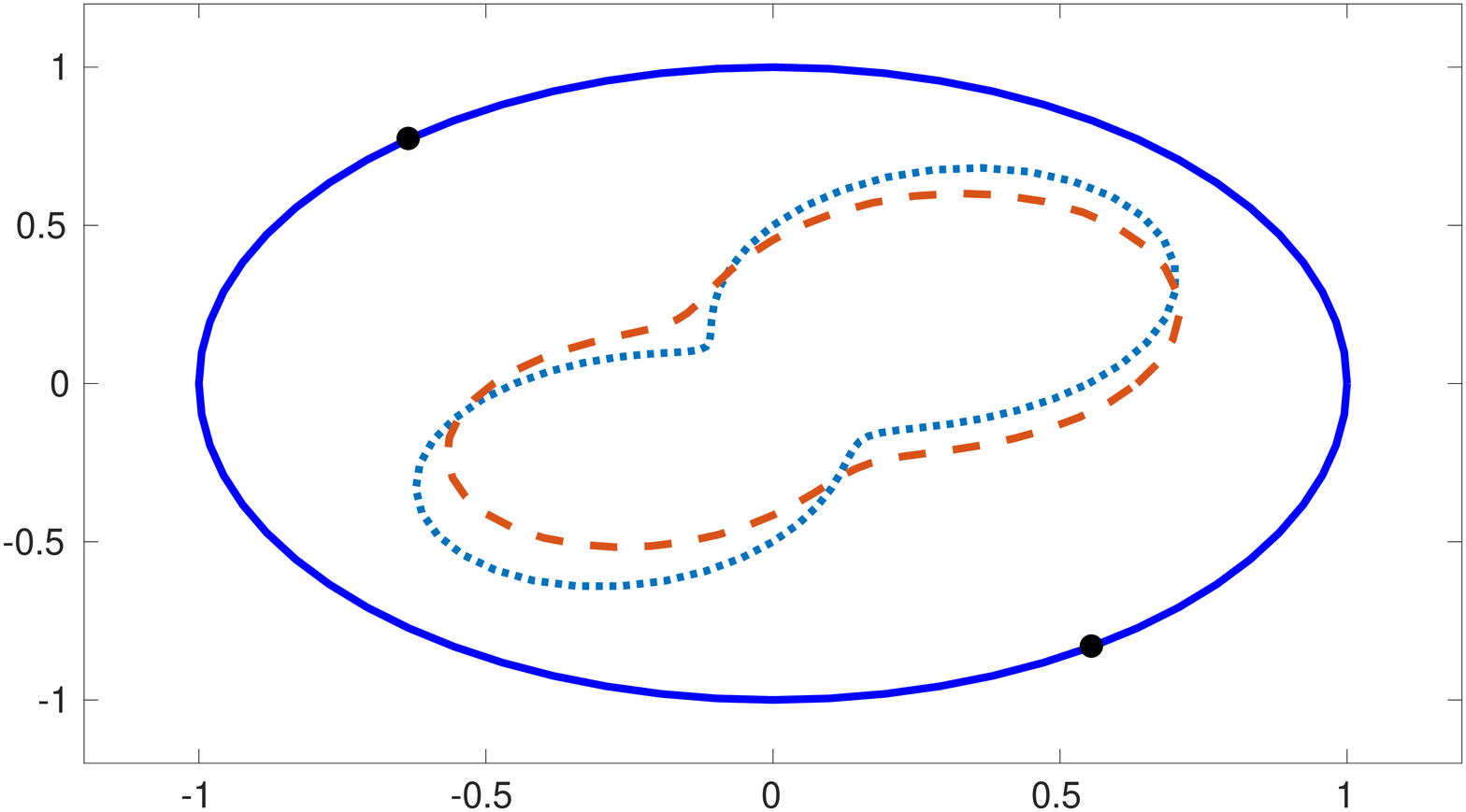}
\end{subfigure}
\begin{subfigure}
  \centering
\includegraphics[trim = .7cm .8cm .2cm .3cm, clip=true,height=5.5cm,width=5.75cm]{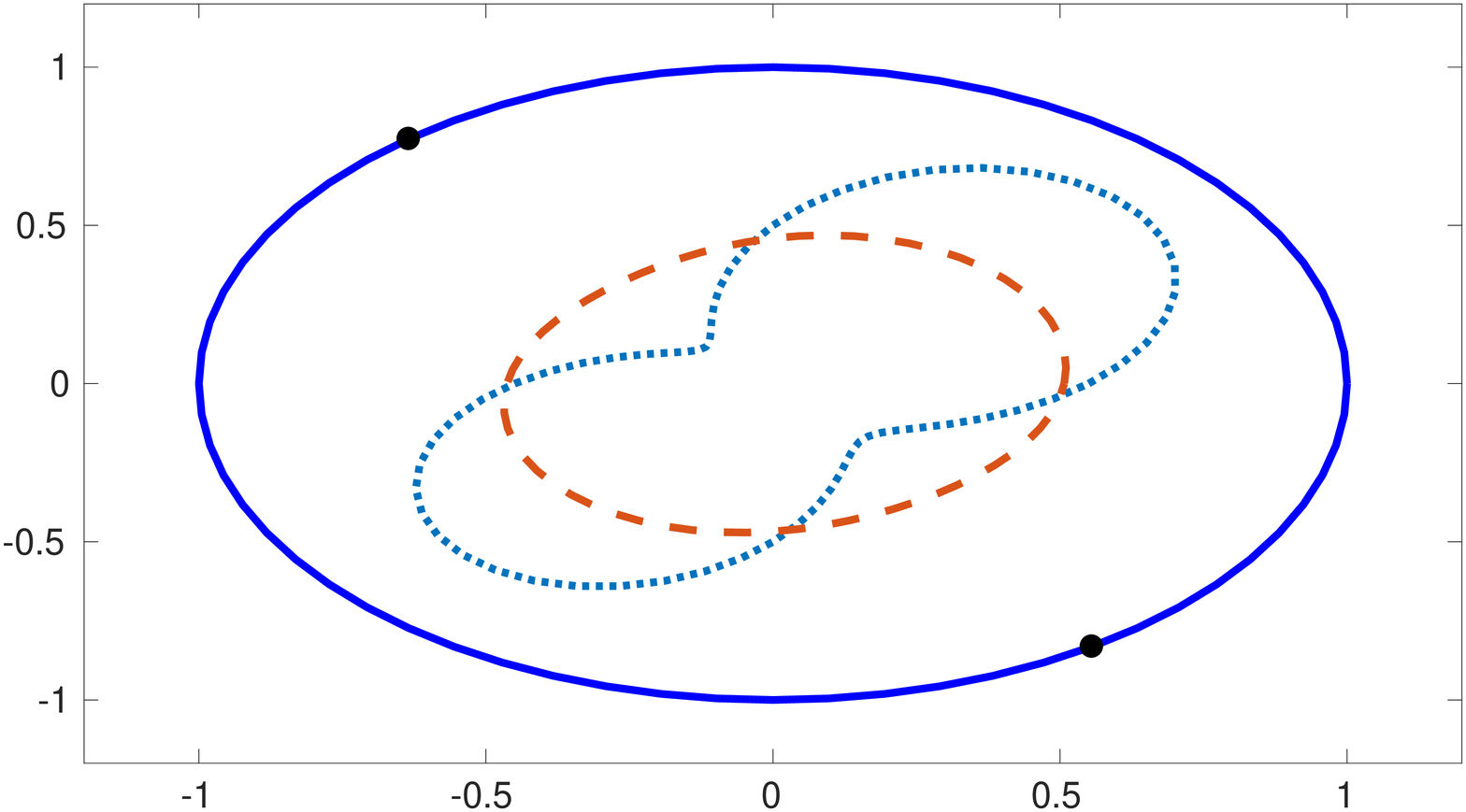}
\end{subfigure}
\vskip-18pt
\caption{\small Reconstructions for $E_{2b}$ with data from $[T_0,T]$ with
$\;T_0=0$, $\,T_0=0.25$, $\,T_0=0.5$,
$\;\alpha = 0.9$.}
\label{fig_cut}
\end{figure}

Figure~\ref{e2b_cut_left} shows how this loss is greater for smaller $\alpha$
as should be expected from the above.
\begin{figure}[th!]\label{e2b_cut_left}
	\center
\begin{subfigure}
  \centering
\includegraphics[trim = .5cm .5cm .5cm .5cm, clip=true,height=5.5cm,width=5.75cm]
		{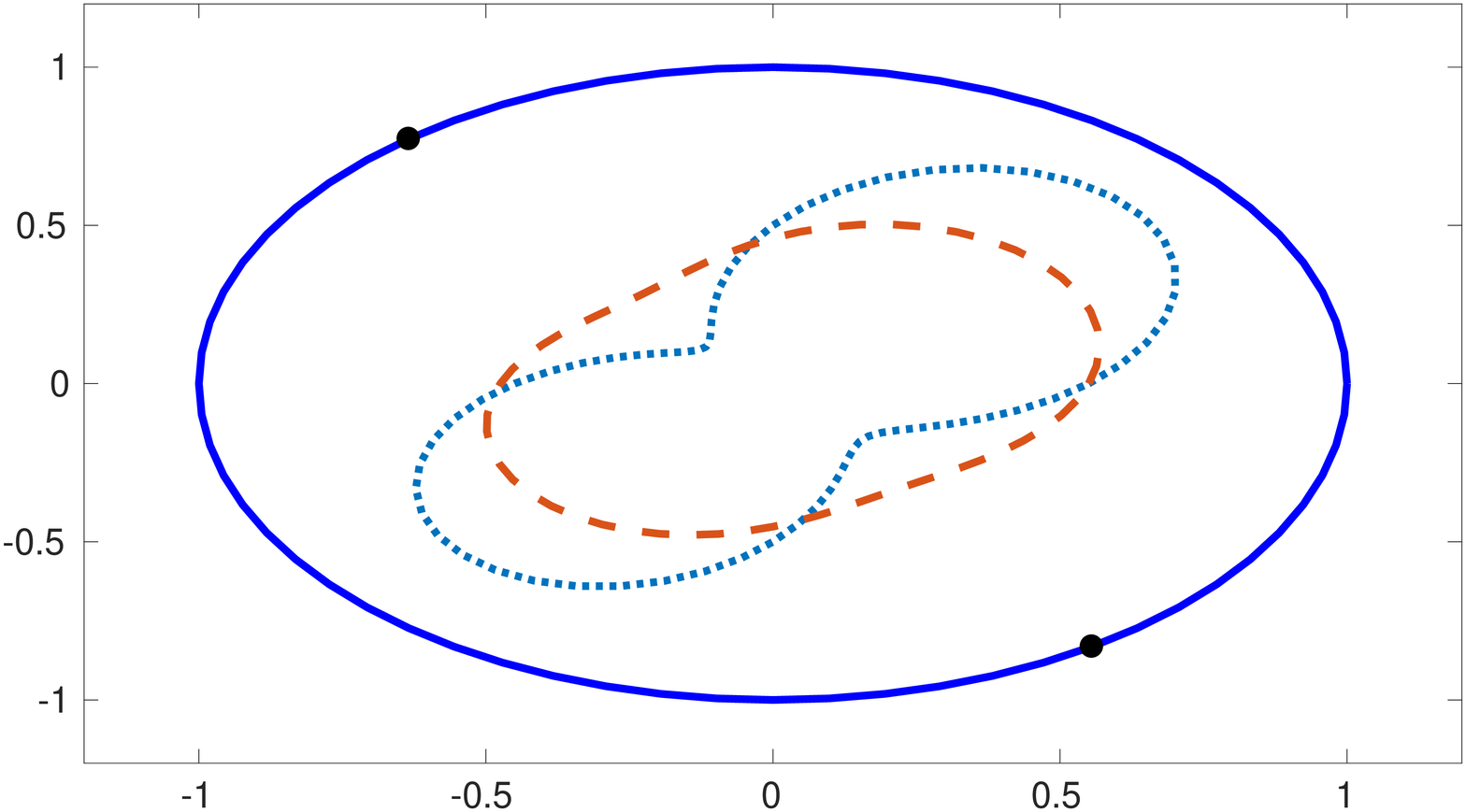}
\end{subfigure}
\begin{subfigure}
  \centering
\includegraphics[trim = .5cm .5cm .5cm .5cm, clip=true,height=5.5cm,width=5.75cm]
		{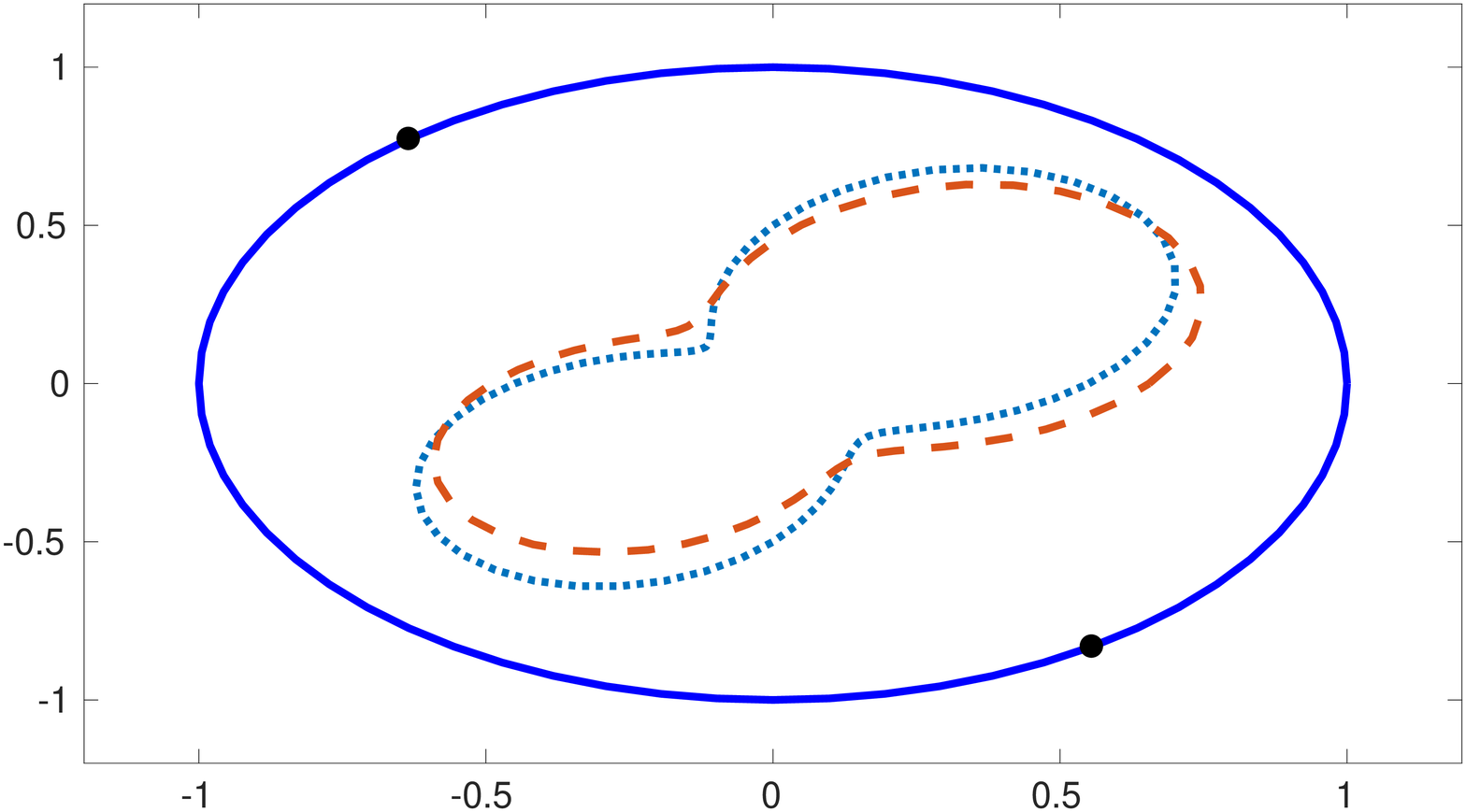}
\end{subfigure}
\vskip-15pt
\caption{\small Results of experiments $E_{2b}$ with data from $[0.25,1]$, $\delta=1\%$ and $\alpha=0.5$ (left), 
$\alpha=1$ (right).}
\label{e2b_cut_left}
\end{figure}

Figure~\ref{e2b_cut_right} shows that when larger time values are missing
the effect is greater for larger $\alpha$ and in particular, for the heat
equation. This is again consistent with the above analysis
and the fact that although the fractional diffusion takes longer to reach
steady state, the later stages of the transient phase contains little
information that can be used to reconstruct the source domain.
\begin{figure}[th!]\label{e2b_cut_right}
	\center
\begin{subfigure}
  \centering
\includegraphics[trim = .5cm .5cm .5cm .8cm, clip=true,height=5.5cm,width=5.75cm]
		{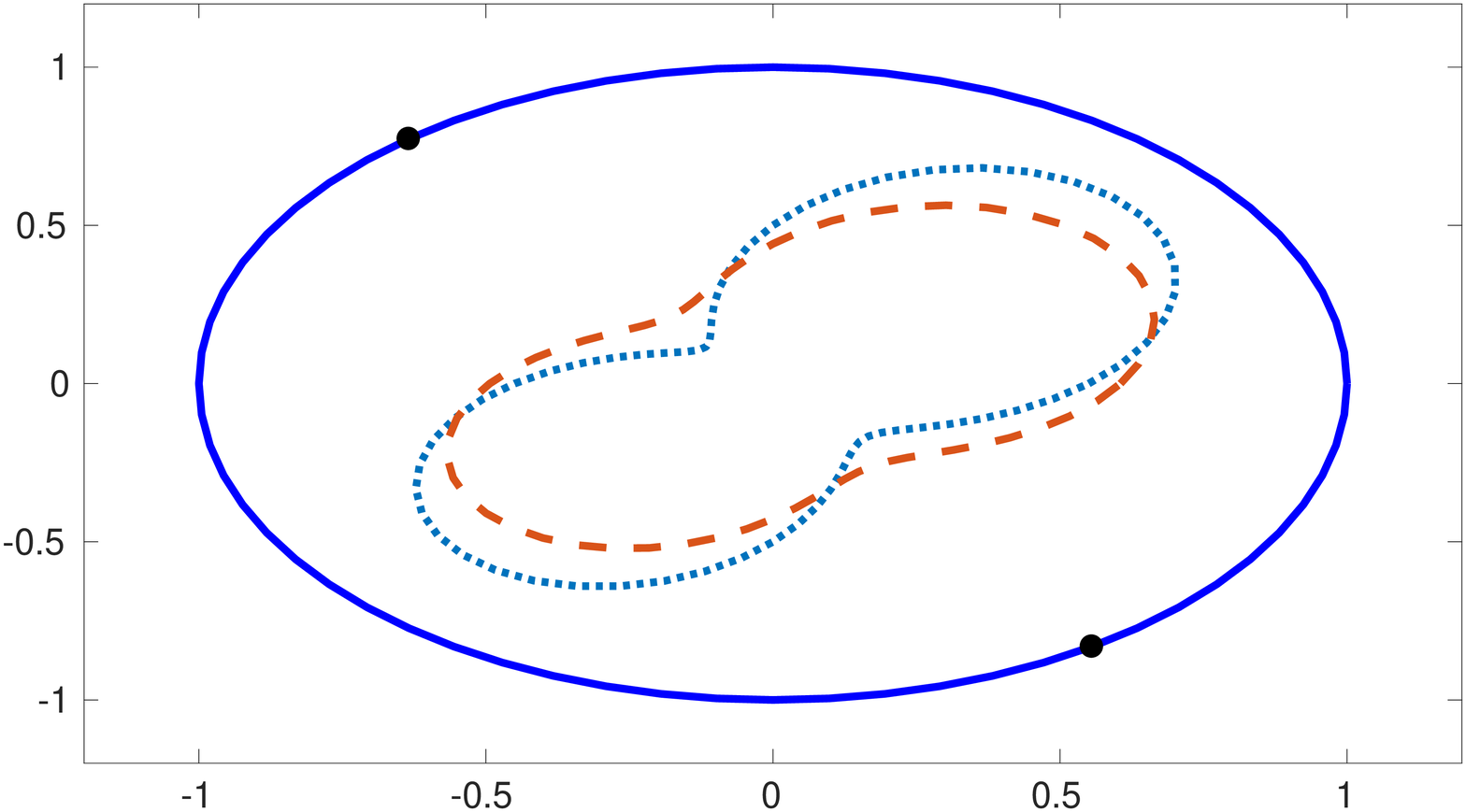}
\end{subfigure}
\begin{subfigure}
  \centering
\includegraphics[trim = .5cm .5cm .5cm .8cm, clip=true,height=5.5cm,width=5.75cm]
		{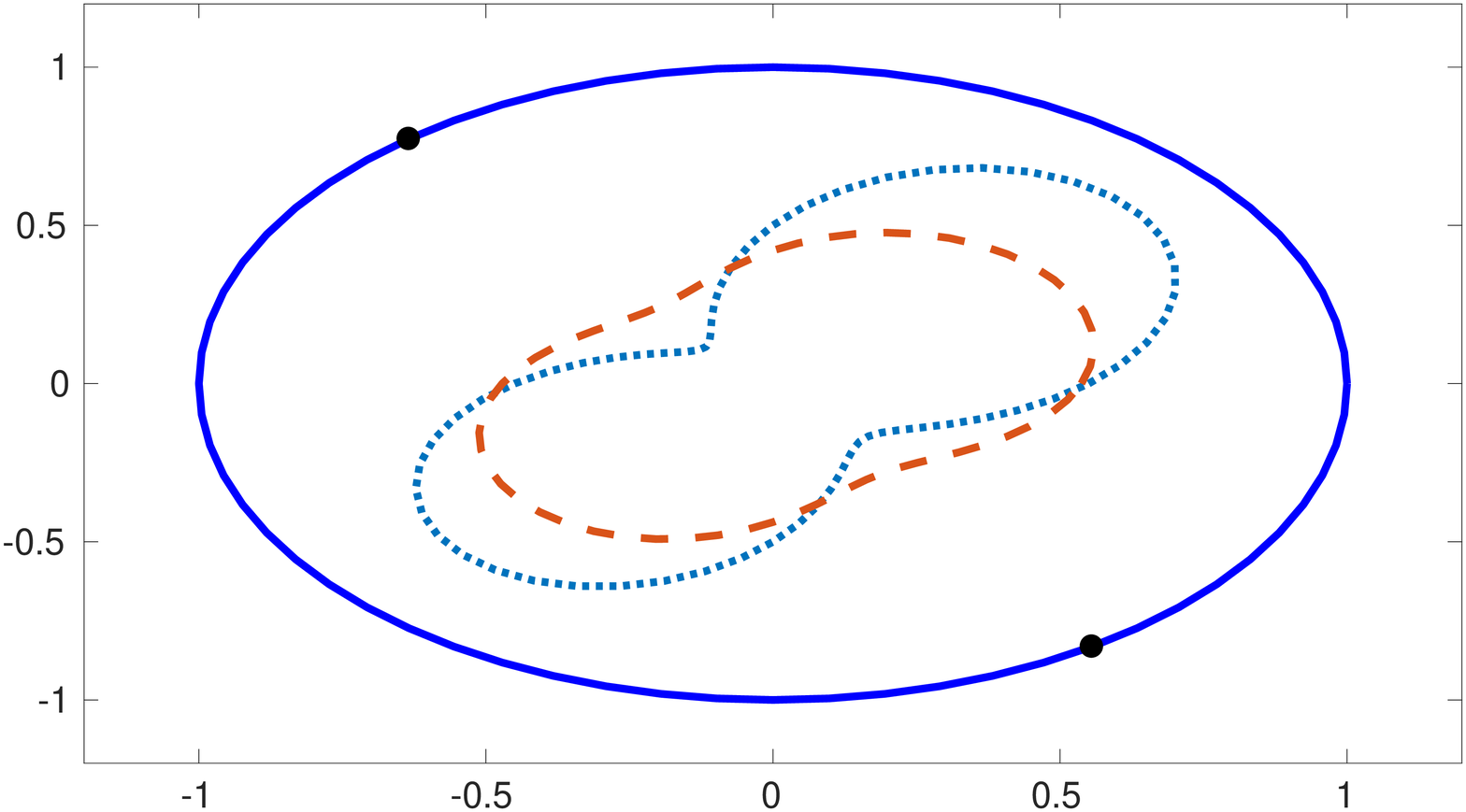}
\end{subfigure}
\vskip-15pt
\caption{\small Results of experiments $E_{2b}$ with data from $[0,0.15]$, $\delta=1\%$ and $\alpha=0.5$ (left), 
$\alpha=1$ (right).}
\label{e2b_cut_right}
\end{figure}

The explanation is clear from \eqref{eqn:spectral_rep}
and perhaps more apparent with the heat equation and the resulting exponential
function $E_{1,1}$ although the identical argument applies to the
Mittag-Leffler function $E_{\alpha,1}$ albeit to a slightly different degree.
For the term $e^{-\lambda_n t}$ to remain sufficiently large to contain
extractable information we require the argument $\lambda_n t$ to be
sufficiently small.
If $\lambda_n < \Lambda$ and $t>T_0$ then
$e^{-\lambda_n t} < e^{-\Lambda T_0} < \epsilon$ for
$\Lambda < -\ln(\epsilon)/T_0$ showing that for a given $\epsilon$
and value $T_0$ we are restricted to a maximum $\Lambda$;
that is we cannot effectively use the $n^{\rm th}$ eigenfunction mode
in equation~\eqref{eqn:spectral_rep} if $\lambda_n>\Lambda$.

In summary, the optimal time-measurement intervals for recovering the
source support $D$ in \eqref{FDE} depend strongly on $\alpha$.
Taking small initial time steps is advantageous in all cases but particularly
important the smaller the value of $\alpha$.

\subsection{More than two measurement points}

We should expect superior reconstructions with a greater number
of observation points since we have additional data for which
to average out measurement error.
However, \eqref{eqn:m-theta_cond} shows much more is possible since
we see that if the difference
$\theta_i-\theta_j$
is near to a rational number $\frac{p}{r}$ times $\pi$ with some $r\leq M$,
then the $r^{\rm th}$ mode will be expressed very poorly from this combination.
For a given $M$, the more observation points taken, the greater the
opportunity to avoid this situation.
This allows an often significant increase in the resulting singular values
and correspondingly a better inversion of $G'$ and hence of the reconstruction.

In experiment $E_{2c}$, we use four observation points.
\begin{equation*} 
\begin{aligned}
{ } &q(\theta)=0.5+0.05\cos{\theta}+0.3\sin{2\theta}, \\
E_{2c}:\qquad 
&\theta_1=\frac{23}{32}\pi,\ \ 
\theta_2=\frac{57}{32}\pi,\ \ 
\theta_3=\frac{1}{4}\pi,\ \ 
\theta_4=\frac{39}{32}\pi, \\
&\beta=3\times10^{-2},\ \delta = 1\%, \ \epsilon=\delta/10.
\end{aligned}
\end{equation*} 
\begin{figure}[th!]\label{e2_multiple_observations}
	\center
\begin{subfigure}
  \centering
\includegraphics[trim = .5cm .5cm .5cm .3cm, clip=true,height=5.5cm,width=5.75cm]
		{e2b.eps}
\end{subfigure}
\begin{subfigure}
  \centering
\includegraphics[trim = .2cm .7cm .5cm .2cm, clip=true,height=5.5cm,width=5.75cm]
		{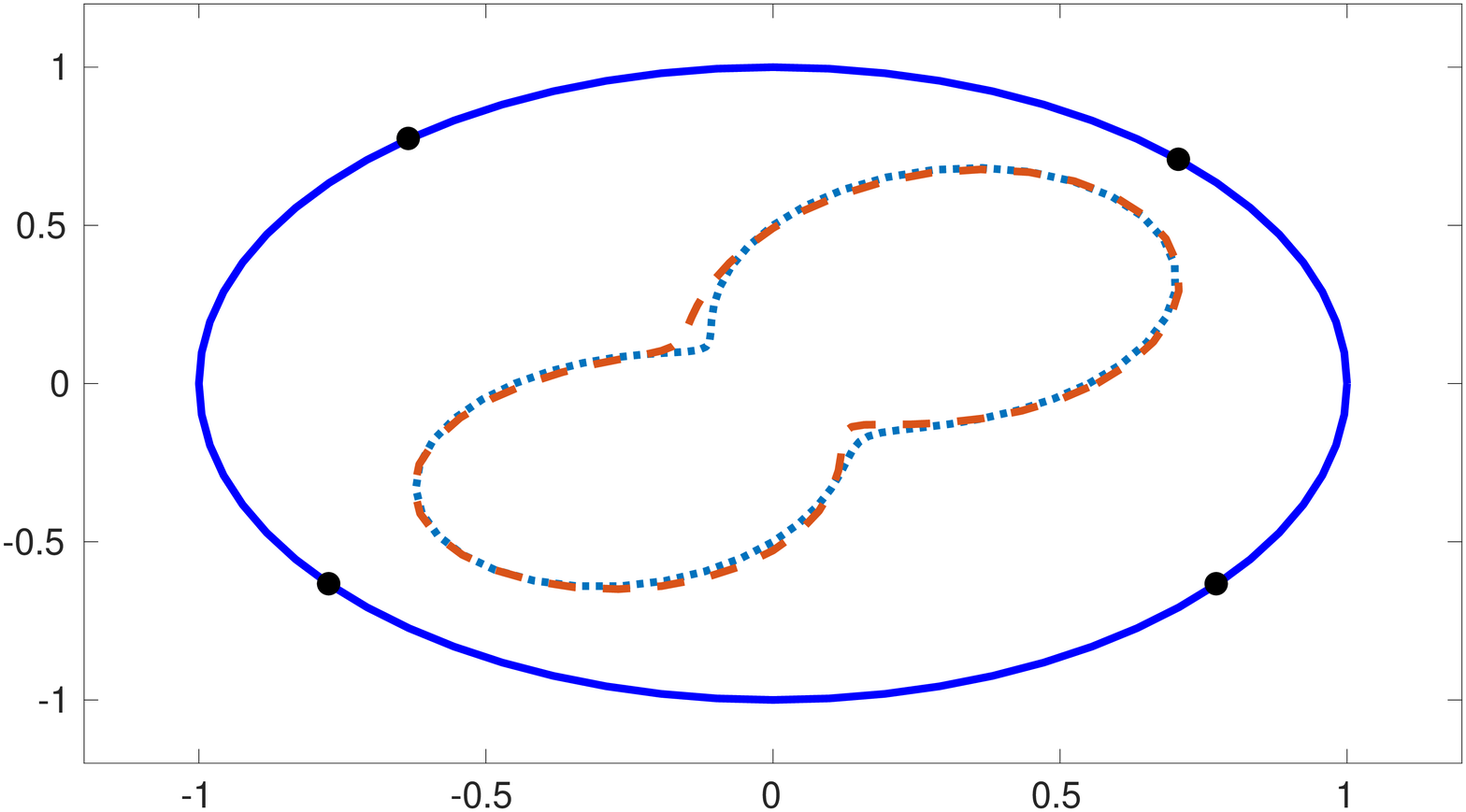}
\end{subfigure}
\vskip-15pt
\caption{\small Results of experiments $E_{2b}$ (left) and $E_{2c}$ (right) with $\alpha=0.9$.}
\label{e2_multiple_observations}
\end{figure}

The result is shown in Figure~\ref{e2_multiple_observations} and by comparison
with Figure~\ref{e_diff_alpha} it is a considerable improvement
over taking just two measurement points.

\section*{Acknowledgment}
Both authors were supported by the
National Science Foundation through award DMS-1620138.
The second author was also supported by
the Finnish Centre of Excellence in Inverse Problems Research 
through project 284715.

\bibliographystyle{plain}
\bibliography{fractional_discontinous_domain}
	
\end{document}